\documentclass{amsart}
\usepackage[margin=1in]{geometry}
\usepackage{amsmath, amsfonts, amssymb, amsthm}
\usepackage{url}
\usepackage{hyperref}
\usepackage{color}
\usepackage{dsfont}

\newtheorem{thm}{Theorem}[section]
\newtheorem{prop}[thm]{Proposition}

\newtheorem{cor}[thm]{Corollary}

\newtheorem{lemma}[thm]{Lemma}

\newtheorem{preremark}[thm]{Remark}
\newenvironment{remark}{\begin{preremark}\rm}{\medskip \end{preremark}}

\numberwithin{equation}{section}

\newcommand{\R}{\mathbb R}
\newcommand{\T}{\mathbb T}
\newcommand{\N}{\mathbb N}

\newcommand{\one}{\mathds 1}
\newcommand{\eps}{\varepsilon}
\newcommand{\dd} {\; \mathrm{d}}
\newcommand{\vv}{\langle v \rangle}

\DeclareMathOperator{\supp}{supp}

\DeclareMathOperator{\PV}{PV}

\title{Solutions to the non-cutoff Boltzmann equation uniformly near a Maxwellian}
\author{Luis Silvestre}
\address[L.~Silvestre]{Mathematics Department, University of Chicago,
  Chicago, Illinois 60637, USA} \email{luis@math.uchicago.edu}
\author{Stanley Snelson}
\address[S.~Snelson]{Department of Mathematical Sciences, Florida Institute of Technology,
  Melbourne, FL 32901, USA} \email{ssnelson@fit.edu}

\thanks{Luis Silvestre is supported by NSF grants 2054888 and 1764285. Stanley Snelson is supported by a Collaboration Grant from the Simons Foundation.}

\begin{document}

\begin{abstract}
The purpose of this paper is to show how the combination of the well-known results for convergence to equilibrium and conditional regularity, in addition to a short-time existence result, lead to a quick proof of the existence of global smooth solutions for the non cutoff Boltzmann equation when the initial data is close to equilibrium. We include a short-time existence result for polynomially-weighted $L^\infty$ initial data. From this, we deduce that if the initial data is sufficiently close to a Maxwellian in this norm, then a smooth solution exists globally in time.
\end{abstract}
\maketitle

\section{Introduction}
We study the inhomogeneous Boltzmann equation without cutoff:
\begin{equation}\label{e:boltzmann}
f_t + v \cdot \nabla_x f = Q(f,f).
\end{equation}
Here $f : [0,\infty) \times \T^d \times \R^d \to [0,\infty)$ is a nonnegative function that solves the equation. We consider the problem periodic in space. For functions $f,g:\R^d\to\R$, the collision operator is defined by
\[ Q(f,g) = \int_{\R^d} \int_{\mathbb S^{d-1}} B(v-v_*,\sigma) [f(v_*') g(v') - f(v_*) g(v)] \dd \sigma \dd v_*,\]
where $v,v_*$ are post-collisional velocities, and $v', v_*'$ are the pre-collisional velocities given by
\[ v' = \frac {v+v_*} 2 + \frac{|v-v_*|} 2 \sigma, \quad v_*' = \frac {v+v_*} 2 - \frac{|v-v_*|} 2 \sigma.\]
We work with the standard noncutoff collision kernel of the form
\[ B(v-v_*,\sigma) = |v-v_*|^\gamma b(\cos\theta),\]
for some $\gamma > - d$, where $\theta$ is the deviation angle between $v$ and $v'$:
\[ \cos\theta = \frac{v-v_*}{|v-v_*|} \cdot \sigma,\] 
and the angular cross-section $b$ has the asymptotics $b(\cos\theta) \approx \theta^{-(d-1)-2s}$ as $\theta \to 0$, for some $s\in (0,1)$. In this paper, we consider parameters $\gamma$ and $s$ such that $\gamma + 2s\in [0,2]$. Thus, the following inequality summarizes the non-cutoff assumptions on the collision kernel $B$.
\begin{equation} \label{e:collision-kernel-assumptions}
C_b^{-1} |v-v_\ast|^{\gamma} |\sin(\theta/2)|^{-d-2s+1} \one_{\cos \theta > 0} \leq B(v-v_\ast,\cos \theta) \leq C_b |v-v_\ast|^{\gamma} |\sin(\theta/2)|^{-d+1-2s}.
\end{equation}
Recall that we can modify the angular cross-section $b$, provided that $b(\theta) + b(\theta+\pi)$ stays the same, without affecting the collision operator $Q$. It is common to make $b=0$ when $\cos(\theta) < 0$. In \cite{silvestre-boltzmann2016}, the choice $B(v-v_\ast,\cos \theta) \approx |v-v'|^{-d+1-2s} |v-v_\ast'|^{1+2s+\gamma} |v-v_\ast|^{-d-2}$ is proposed to simplify some of the computations.

Let $M : \R^d \to [0,+\infty)$ be the Maxwellian distribution, which is a stationary solution to \eqref{e:boltzmann} that is constant in $t$ and $x$. For concreteness, let us take the usual normalized Maxwellian: $M(v) = (2\pi)^{-d/2} e^{-|v|^2/2}$. The analysis in this paper works around any nonzero Maxwellian.

We consider solutions of the form $f = M+\tilde f$, and $\tilde f$ will be taken to be small. The function $\tilde f$ satisfies the equation
\begin{equation} \label{e:boltzmann-perturbative}
\tilde f_t + v \cdot \nabla_x \tilde f = Q(M+\tilde f, \tilde f) + Q(\tilde f, M).
\end{equation} 
It is natural to choose $M$ with the same mass, momentum and energy as $f$. That is,
\begin{align*}
\iint_{\T^d \times \R^d} \tilde f(t,x,v) \dd v \dd x = 0, \\
\iint_{\T^d \times \R^d} \tilde f(t,x,v) v \dd v \dd x = 0, \\
\iint_{\T^d \times \R^d} \tilde f(t,x,v) |v|^2 \dd v \dd x = 0.
\end{align*}
Since mass, momentum and energy are conserved in time, if these identities hold at time $t=0$, they will hold for all time.

Note that the function $\tilde f$ in \eqref{e:boltzmann-perturbative} may take both positive and negative values. The main results of this paper concern small solutions of the equation \eqref{e:boltzmann-perturbative}. We will omit the tilde in $\tilde f$ from now on, even when we refer to solutions to \eqref{e:boltzmann-perturbative} instead of \eqref{e:boltzmann}. We also introduce the notation $\langle \cdot\rangle = \sqrt{1+|\cdot|^2}$, which will be used throughout the paper.

We split our main result into two main theorems depending on the values of $s$ and $\gamma$. The first one concerns the case of hard potentials $\gamma > 0$. The second one is for moderately soft potentials, which corresponds to those values of $\gamma \leq 0$ so that $\gamma+2s \geq 0$.

\begin{thm} \label{t:main}
Assume $\gamma > 0$ and $\gamma + 2s \in [0,2]$, and let $q > 0$ be sufficiently large, depending on $\gamma$, $s$, and the constant $C_b$ in \eqref{e:collision-kernel-assumptions}. Given any $\eps_0 > 0$, there exists $\eps_1>0$ (presumably much smaller than $\eps_0$) so that if the initial data $f_0 : \T^3 \times \R^3 \to \R$ satisfies
\begin{equation} \label{e:initial_smallness}
 |f_0(x,v)| < \eps_1 \vv^{-q},
 \end{equation} 
then the equation \eqref{e:boltzmann-perturbative} has a global solution $f$, with initial data $f_0$, that satisfies
\begin{equation} \label{e:aim}
|f(t,x,v)| < \eps_0 \vv^{-q},
\end{equation}
for all $t,x,v \in [0,\infty) \times \T^3 \times \R^3$.

This global solution is $C^\infty$ and decays rapidly for large velocities. More precisely, for any multi-index $\alpha$ (involving derivatives in time, space and/or velocity), $k \in \mathbb N$ and $\tau >0$, 
\[
|\vv^k \partial^\alpha f(t,x,v)| \leq C(\alpha,k,\tau) \ \text{ for all } (t,x,v) \in [\tau,+\infty) \times \T^3 \times \R^3.
\]
Here, the upper bound $C(\alpha,k,\tau)$ depends on $\alpha$, $k$, $\tau$, $\gamma$ and $s$ only.
\end{thm}

\begin{thm} \label{t:main2}
Let $\gamma \leq 0$, and assume $\gamma + 2s \in [0,2]$. There exists a sufficiently large exponent $q_0$ (depending on $\gamma$, $s$, and $C_b$) so that the following statement is true: Given $\eps_0>0$, and the sequence of numbers $N_q$ for $q=1,2,3,\dots$ there exists $\eps_1>0$, such that for any initial data $f_0:\T^3\times\R^3\to \R$ satisfying, for all $q>0$,  
\begin{equation}
 |f_0(x,v)| < N_q \vv^{-q}, 
 \end{equation}
and also
\begin{equation}
 |f_0(x,v)| < \eps_1 \vv^{-q_0}, 
 \end{equation}
then \eqref{e:boltzmann-perturbative} has a global solution $f$ with initial data $f_0$, that satisfies for some family of constants $N_q'$,
\begin{equation}
|f(t,x,v)| < N_q' \vv^{-q},
\end{equation}
with some $N_{q_0}' < \eps_0$.

Here $\eps_1$ depends on $\gamma$, $s$, $\eps_0$, and the numbers $N_q$. The value of $q_0$ depends on $\gamma$ and $s$ only.

This global solution is $C^\infty$ and decays rapidly for large velocities. More precisely, for any multi-index $\alpha$ (involving derivatives in time, space and/or velocity), $k \in \mathbb N$ and $\tau >0$, 
\[
|\vv^k \partial^\alpha f(t,x,v)| \leq C(\alpha,k,\tau) \ \text{ for all } (t,x,v) \in [\tau,+\infty) \times \T^3 \times \R^3.
\]
Here, the upper bound $C(\alpha,k,\tau)$ depends on $\alpha$, $k$, $\tau$, $\gamma$, $s$ and the initial upper bounds $N_q$.
\end{thm}

The following Corollary of Theorem \ref{t:main2} is perhaps easier to comprehend.
\begin{cor}\label{c:main2}
Let $\gamma \leq 0$, and assume $\gamma + 2s \in [0,2]$. Let $\varphi: [0,\infty) \to [0,\infty)$ be a function so that $\varphi(r)/r^k \to 0$ as $r \to \infty$ for every $k > 0$. For any $\eps_0>0$ and $q \geq 0$, there exists $\eps_1>0$, such that for any initial data $f_0:\T^3\times\R^3\to\R$ satisfying 
\begin{equation}
 |f_0(x,v)| < \eps_1 \varphi(|v|)
 \end{equation}
then \eqref{e:boltzmann-perturbative} has a global smooth solution $f$, rapidly decaying for $|v|\to \infty$, with initial data $f_0$, that satisfies
\[ |f(t,x,v)| < \eps_0 \vv^{-q}, \]
for all $t,x,v \in [0,\infty) \times \T^3 \times \R^3$.
\end{cor}

The first results of global existance of solutions to the non-cutoff Boltzmann equation for initial data near equilibrium were given independently in \cite{gressmanstrain2011} and \cite{amuxy2011_hard,amuxy2012_soft}. A key development leading to the result in \cite{gressmanstrain2011} is an anisotropic distance and sharp coercivity estimates that capture the right asymptotics for large velocities. This coercivity is with respect to an anisotropic fractional Sobolev norm (see \eqref{e:Nsgamma-Hs}) that also plays a role in the analysis of solutions that are not necessarily close to equilibrium (see \cite{gressman-strain2011sharp}).  More recently, other global existence results have been obtained measuring the closedness between the initial data and a Maxwellian with different norms \cite{duan2019global,herau2020regularization,alonso2018non,alonso2020giorgi,zhang2020global}. In \cite{herau2020regularization} and \cite{alonso2018non}, the authors use Sobolev norms with polynomial weights. A consequence of the result of \cite{herau2020regularization} is the improvement of the decay rate to equilibrium from \cite{desvillettes2005global}. In \cite{alonso2020giorgi}, the authors study a small perturbation of the Maxwellian in $L^\infty$, with a polynomially decaying weight. The analysis is based on an $L^\infty$ estimate by De Giorgi iteration. We also point out an earlier result in this direction for the Landau equation in \cite{kim2020Landau}.

Given these precedent results, one can argue that Theorems \ref{t:main} and \ref{t:main2} in this paper are not too surprising. In particular, the result in \cite{alonso2020giorgi} contains Theorem \ref{t:main} in the case of $\gamma \in [0,1]$. Theorem \ref{t:main2} appears to be new, extending the result of \cite{alonso2020giorgi} to moderately soft potentials. Our main motivation for this work is not so much to come up with a better result than the ones in the literature concerning global existence of smooth solutions for initial data near a Maxwellian, but to show how this type of result can be quickly derived from the combination of the following three ingredients:
\begin{enumerate}
	\item \textbf{The convergence to equilibrium}. In a celebrated result by Desvillettes and Villani \cite{desvillettes2005global}, it is proved that solutions converge to the Maxwellian as $t \to \infty$, conditional to uniform regularity estimates and a uniform lower bound by a fixed Maxwellian.
	\item \textbf{Conditional regularity}. In a series of recent works \cite{silvestre-boltzmann2016,imbert-silvestre-whi2020,imbert2018schauder,imbert-mouhot-silvestre-lowerbound2020,imbert-mouhot-silvestre-decay2020,imbert2019global,imbert-silvestre-survey2020}, global regularity estimates and lower bounds are obtained conditional only to certain macroscopic bounds.
	\item \textbf{A short-time existence result}. Here, the time of existence of the solution should depend on some distance between the initial data and a Maxwellian.
\end{enumerate}

The first two items in this list, the convergence to equilibrium and the conditional regularity estimates, apply to arbitrary solutions away from equilibrium. Neither of these results was meant to be applied to the near-equilibrium regime. However, if we do apply them to solutions that are near a Maxwellian, then their conditional assumptions are automatically satisfied, and they simplify considerably the problem of establishing the global existence of solutions. Indeed, the proof in this paper is quite short, at the expense of applying these two elaborate theorems from the literature.

Even though there are several documented results about global solutions for initial data near equilibrium, there are suprisingly few explicit results about the short-time existence for arbitrary initial data. The first such result that we know of is in \cite{amuxy2011qualitative}. It requires the initial data to have Gaussian decay for large velocities, which makes it difficult to apply in practice. In \cite{morimoto2015polynomial} and \cite{henderson2020polynomial}, the authors obtain a short-time existence result for initial data in $H^6$ and $H^5$ respectively, with a polynomial weight, in the case of soft potentials: $\gamma \leq 0$.

In this paper, we include a proof of short-time existence for solutions in Section \ref{s:short-time}. This proof comprises the bulk of this paper. Once we have the three ingredients mentioned above, proving Theorems \ref{t:main} and \ref{t:main2} becomes practically trivial. Using the result of \cite{henderson2020polynomial}, we could have shortened the document considerably, but we wanted to be able to address the case of hard potentials, $\gamma > 0$, as well. The proof for short-time existence we include here is relatively minimalistic, and it is not meant to be applied to initial data that is far from equilibrium. There is certainly room (and, arguably, need) for further research into short-time existence results for the Boltzmann equation.

\begin{remark}
The regularity estimates for the solution $f$ in Theorems \ref{t:main} and \ref{t:main2} are a direct consequence of the result in \cite{imbert2019global} (see Theorem \ref{t:reg} below). The solutions are uniformly smooth for $t \in [\tau,+\infty)$, for any $\tau > 0$. Interpolating these regularity estimates with the upper bound $|f| \leq \eps_0 \vv^{-q}$ one can directly deduce that, by picking $\eps_0$ small, the norm $\|f(t,\cdot,\cdot)\|_{W^{k,p}_m(\T^3 \times \R^3)}$ can be made arbitrarily small, uniformly for $t \in [\tau,+\infty)$, for any values of $k \geq 0$, $p \in [1,\infty]$, and $m \geq 0$. Thus, the conditional regularity estimates allow us to transfer our initial smallness condition with respect to the norm $L^\infty_q$ into stronger norms.
\end{remark}

\subsection{Strategy for the proof}

The proofs of Theorems \ref{t:main} and \ref{t:main2} follow quickly by combining the trend to equilibrium, conditional regularity, and short-time existence.

Any short-time existence result gives us a solution in an interval of time $[0,T]$, for some $T$ depending on the initial data. Since any Maxwellian is a stationary solution, it is natural to estimate the time $T$ depending on some kind of distance between $f_0$ and a given Maxwellian $M$. Depending on how we set up our short-time existence result, we may utilize different norms for $\|f_0 - M\|$. The proof we provide in Section \ref{s:short-time} uses a polynomially weighted $L^\infty$ norm. We made this choice of norm so that our main theorems match and extend one of the latest (and arguably strongest) results in the current literature: \cite{alonso2020giorgi}.

There is one condition that the norm used in the short-time existence result must satisfy for our proof to work. The smallness of $\|f_0 - M\|$ must imply the hydrodynamic bounds that are required for the conditional regularity result (given below in \eqref{e:hydro}). It is hard to imagine that this would ever be a problem. Every result in the current literature for solutions near a Maxwellian imposes a smallness condition that is stronger than the hydrodynamic bounds in \eqref{e:hydro}.

The main idea for the proofs of Theorem \ref{t:main} and \ref{t:main2} is the following. We know that there is a solution for certain amount of time. Before this solution ceases to exist, it will first invalidate the inequality $\|f(t,\cdot) - M\| < \eps_0$ at certain time $T_0$. From the conditional regularity results, we know that the solution is smooth and bounded below by a Maxwellian in $[\tau,T_0]$, for any given $\tau>0$. These estimates do not depend on $T_0$. The trend to equilibrium result tells us that $\|f(T_0,\cdot) - M\| < \eps_0$ for sure if $T_0$ is too large, leading to a contradiction. The proof finishes immediately by picking $\eps_1$ small enough so that the short-time existence result ensures the solution exists for a long enough interval of time.

The fact that in our Theorems \ref{t:main} and \ref{t:main2}, we measure the closeness between the initial data $f_0$ and the Maxwellian $M$ with respect to a weighted $L^\infty$ norm depends exclusively on the type of short-time existence result we use. A different short-time existence result, with different conditions on the initial data, would automatically lead to a different global existence result, for a different way to measure the distance between $f_0$ and $M$. The method presented in this paper reduces any future attempt to prove the existence of global solutions near equilibrium, to establishig a new short-time existence result.


We do not address the uniqueness of solutions in this paper. Uniqueness is a local property. It is something that depends exclusively on the short-time result as well. We prove local existence in $L_q^\infty$, but not uniqueness, in Section \ref{s:short-time}.

We should point out that most previous works concerning the global well posedness of the Boltzmann equation near equilibrium contain some form of short-time existence somewhere inside their proofs, in one way or another. With the approach we suggest in this paper, we cannot get around that part of the proof. But we save ourselves from redoing anything else.

\subsection{Further results on global regularity near equilibrium}

One advantage of having a quick proof of global existence near equilibrium as a consequence of the three ingredients mentioned above is that new results concerning convergence to equilibrium, conditional regularity, or short-time existence, would automatically translate into new results on global existence.

Any result on short-time existence of (smooth) solutions whose time of existence depends on some distance between the inital data and a Maxwellian, would immediately imply a global existence result when that distance is sufficiently small. Thus, if one wanted to extend Theorems \ref{t:main} and \ref{t:main2} to norms other than polynomially-weighted $L^\infty$, we would need to develop a suitable short-time existence result only. One caveat is on the decay of the solutions as $|v| \to \infty$. In this paper we use techniques from \cite{imbert-mouhot-silvestre-decay2020} to propagate polynomially decaying upper bounds. Other rates of decay would require different bounds.

The reason we require $\gamma+2s \in [0,2]$ in Theorems \ref{t:main} and \ref{t:main2} is because of the same requirement in the conditional regularity result from \cite{imbert2019global}. The restriction $\gamma+2s \in [0,2]$ plays a strong role in establishing the $L^\infty$ estimates in \cite{silvestre-boltzmann2016}, that are applied in \cite{imbert2019global}. The solutions we work with in this paper are bounded by construction, so it seems that the assumption $\gamma+2s \geq 0$ should be unnecessary. Presumably, one should be able to establish $C^\infty$ estimates similar to \cite{imbert2019global}, in the full range of values of $\gamma$ and $s$, if we add the condition that $f \in L^\infty$ in addition to \eqref{e:hydro}. We are not aware of any result for $\gamma+2s < 0$ in this direction. See \cite{cameron2020velocity} for a result for $\gamma+2s > 2$.

To be more precise, it is conceivable that Theorem \ref{t:reg} can be extended to other values of $\gamma$ and $s$ if the regularity estimates are allowed to depend on $\|f\|_{L^\infty}$ in addition to \eqref{e:hydro}. Such a conditional regularity result could be used to extend the result of Theorems \ref{t:main} and \ref{t:main2} to the same values of $\gamma$ and $s$. The short-time existence result we present in Section \ref{s:short-time} also uses the restriction $\gamma+2s \geq 0$ indirectly through the application of the result in \cite{he2012homogeneous-boltzmann} for the construction of solutions. Note that, for $0>\gamma>\max(-3,-3/2-2s)$, we can use the result in \cite{henderson2020polynomial} to construct the solutions for a short time.

The reason we state Theorems \ref{t:main} and \ref{t:main2} in three dimensions is because we reference a result from 
 \cite{he2012homogeneous-boltzmann} that is stated in three dimensions in the original paper. There seems to be no fundamental difficulty in extending this result, and therefore also the analysis in this paper, to higher dimensions. Naturally, the number of derivatives in Proposition \ref{p:short_time} would change from $4$ to a dimensional-dependent number.

\subsection{Notation}

We denote polynomially-weighted $L^p$ and Sobolev norms by
\[ \|u\|_{L^p_q(\T^d\times\R^d)} := \|\vv^q u\|_{L^p(\T^d\times\R^d)}, \quad \|u\|_{H^k_q(\T^d\times \R^d)} := \|\vv^q u\|_{H^k(\T^d\times\R^d)}, \quad k\in \mathbb N.\]
For functions of the $v$ variable only, the norms $\|u\|_{L^p_q(\R^d)}$ and $\|u\|_{H^k_q(\R^d)}$ are defined analogously. For $s\in (0,1)$ and $[u]_{H^s(\R^d)}$ the standard fractional Sobolev seminorm defined by
\[ [u]_{H^s(\R^d)}^2 = \int_{\R^d}\int_{\R^d} \frac{|u(v) - u(w)|^2}{|v-w|^{d+2s}} \dd w \dd v ,\] 
we define
 \[ \|u\|_{H^{k+s}_q(\R^d)} := \|u\|_{H^k_q(\R^d)} + \sum_{|\alpha|=k} [\vv^q \partial^\alpha u]_{H^s(\R^d)}. \]

\section{Preliminaries}

\subsection{The initial data}
\label{s:weak_solutions}

We work with classical smooth solutions in this paper. The solutions we construct are $C^\infty$ with respect to all variables ($t$, $x$ and $v$) for any positive time. They also decay as $|v| \to \infty$ faster than any algebraic rate. However, this smoothness is not uniform up to the initial time, unless $f_0$ is smooth.

Since we allow initial data $f_0$ that is in a weighted $L^\infty$ space, our solutions may have a discontinuity at time $t=0$. One way to make sense of the initial data is to use the weak formulation of the equation up to $t=0$. We insist that for any smooth test function $\varphi(t,x,v)$ with compact support in $[0,T)\times \T^3\times\R^3$, the following equality holds:
\begin{equation}\label{e:initial-data}
 \iint_{\T^3\times\R^3} f_0(x,v) \varphi(0,x,v) \dd x \dd v = \iiint_{[0,T]\times\T^3\times\R^3} \left[ f(\partial_t + v\cdot \nabla_x )\varphi + \varphi Q(M+f,M+f) \right] \dd v \dd x \dd t.
 \end{equation}
It is not a priori obvious that the last term on the right is well-defined, as the bounds necessary to make pointwise sense of $Q(M+f,M+f)$ may degenerate as $t\to 0$. However, there is a weak form of the collision operator that allows one to make sense of the integral using the smoothness of $\varphi$: with the notations $\varphi = \varphi(t,x,v)$, $\varphi_* = \varphi(t,x,v_*)$, $\varphi' = \varphi(t,x,v')$, $\varphi_*' = \varphi(t,x,v_*')$, and similarly for $f$, one has
\begin{equation}\label{e:weak-collision}
\begin{split} \int_{\R^3} &\varphi  Q(M+f,M+f) \dd v\\
& = \frac 1 2 \int_{\R^3}\iint_{\R^3\times \mathbb S^2}(M(v)+f)(M(v_*)+ f_*) B(v-v_*,\sigma) (\varphi_*' +\varphi'- \varphi_* - \varphi) \dd \sigma \dd v_* \dd v.
\end{split}
 \end{equation}
This weak formulation follows from well-known formal computations (see \cite[Chapter 1, Section 2.3]{villani_review}) that are valid in our setting because $f(t,x,\cdot)$ is Schwartz class for all $t>0$ and $x\in \T^3$. We also have the following estimate (see \cite[Chapter 2, formula (112)]{villani_review}) that ensures the integral in $v$ is well-defined as long as $\varphi \in W^{2,\infty}_v$ and $f\in (L^1_{\gamma_++2})_v$:
\begin{lemma}\label{l:phiQgh}
For $\varphi, g, h: \R^d\to \R$ such that the right-hand side is finite, one has
\begin{equation*}
 \int_{\R^d} \varphi Q(g,h) \dd v \lesssim \|\varphi\|_{W^{2,\infty}(\R^d)} \iint_{\R^d\times\R^d}g(v_*) h(v) |v-v_*|^{1+\gamma} \langle v-v_* \rangle \dd v_* \dd v.
 \end{equation*}
\end{lemma}

Since the test functions $\varphi$ in \eqref{e:initial-data} are smooth up to $t=0$, our polynomially-weighted $L^\infty$ bounds on $f$ ensure that the formula \eqref{e:initial-data} makes sense.



\subsection{Conditional regularity}

The conditional regularity of the Boltzmann equation is a collection of regularity estimates that are based on the assumption that we have a classical solution $f$ to \eqref{e:boltzmann} whose macroscopic hydrodynamic quantities satisfy the following pointwise bounds.
\begin{equation}\label{e:hydro}
\begin{split}
0<m_0\leq \int_{\R^d} f(t,x,v)\dd v \leq M_0,\\
\int_{\R^d} |v|^2 f(t,x,v)\dd v \leq E_0,\\
\int_{\R^d} f(t,x,v) \log f(t,x,v) \dd v \leq H_0.
\end{split}
\end{equation}

The inequalities in \eqref{e:hydro} concern the mass, energy and entropy densities respectively, at every point $(t,x)$. We do not know of any reason why the inequalities \eqref{e:hydro} should hold for general solutions. By analogy with the compressible Euler and Navier Stokes systems, it would make sense to expect that there may exist some solutions where \eqref{e:hydro} does not hold. The conditional regularity result tells us that this is the only way in which the Boltzmann equation can possibly develop a singularity.

The main result of \cite{imbert2019global} provides a priori estimates for derivatives of any order, depending only on the conditions \eqref{e:hydro}.

\begin{thm}\label{t:reg}
Let $f$ be a classical smooth solution to \eqref{e:boltzmann} on $[0,T]\times \T^d\times \R^d$ which decays faster than any algebraic rate as $|v| \to \infty$. Assume that \eqref{e:hydro} holds. Then, for any multi-index $k \in \mathbb N^{1+2d}$, $q>0$, and $\tau\in (0,T)$,
\begin{equation}\label{e:reg-estimates}
\|\vv^q D^k f\|_{L^\infty((\tau,T]\times \T^d\times \R^d)} \leq C_{k,q,\tau}.
\end{equation}

When $\gamma > 0$, the constants $C_{k,q,\tau}$ depend only on $k$, $q$, $\tau$, $s$, $\gamma$, $d$, and the constants in \eqref{e:hydro}. When $\gamma \leq 0$, the constants $C_{k,q,\tau}$ depend additionally on polynomial decay estimates for $f_0$, i.e. on the constants
\[ N_r := \sup_{x,v} \vv^r f_0(x,v), \quad \text{ for each } r\geq 0.\]
\end{thm}

Theorem \ref{t:reg} is an \emph{a priori} estimate for smooth solutions---the key aspect is that the regularity and decay estimates depend on the zeroth-order quantities in the assumption \eqref{e:hydro} and are quantitatively independent of the qualitative assumption of smoothness.

The following simple lemma allows us to apply Theorem \ref{t:reg} in the near-equilibrium context. It says that a function sufficiently close to a Maxwellian automatically satisfies the hydrodynamic assumptions \eqref{e:hydro}.

\begin{lemma}\label{l:hydro-near_equilibrium}
If $M(v)$ is a Maxwellian distribution, and $g:\R^d\to \R$ satisfies $\vv^q g(v) \leq \frac 1 2$ for all $v\in \R^d$, for some $q>d + 2$, and $M+g \geq 0$, then $M+g$ satisfies the assumptions \eqref{e:hydro} with constants $m_0$, $M_0$, $E_0$, and $H_0$ depending only on $q$.
\end{lemma}
\begin{proof} 
Direct calculation.
\end{proof}

\subsection{Trend to equilibrium}

In \cite{desvillettes2005global}, Desvillettes and Villani showed that solutions to the Boltzmann equation \eqref{e:boltzmann} satisfying regularity and non-vacuation conditions that are uniform in $t$, converge to Maxwellians as $t\to \infty$.

\begin{thm}\label{t:global-attractor}
Let $f\geq 0$ be a solution to \eqref{e:boltzmann} on $[0,T]\times \T^d\times \R^d$ satisfying, for a family of positive constants $C_{k,q}$, 
\[ \|f\|_{L^\infty([0,T],H^{k}_{q}(\T^d\times \R^d))} \leq C_{k,q} \quad \text{ for all } k,q\geq 0,\]
and also satisfying the pointwise lower bound
\[ f(t,x,v) \geq K_0 e^{-A_0 |v|^2}, \quad \text{ for all } (t,x,v).\]
Then for any $p>0$ and for any $k,q>0$, there exists $C_p>0$ depending on $d$, $\gamma$, $s$, $A_0$, $K_0$, the constant $C_b$ in \eqref{e:collision-kernel-assumptions}, and $C_{k',q'}$ for sufficiently large $k'$ and $q'$, such that for all $t \in [0,T]$,
\[   \|f(t,\cdot,\cdot)- M\|_{H^{k}_{q}(\T^d\times \R^d)} \leq C_p t^{-p},\]
where $M$ is the Maxwellian with the same total mass, momentum, and energy as $f$.
\end{thm}

The convergence rate given in \cite{desvillettes2005global} is faster than any polynomial rate, but not explicitly exponential. However, after the analysis in \cite{herau2020regularization} or \cite{alonso2020giorgi}, we now know that the decay rate is actually exponential with the same hypothesis as in Theorem \ref{t:global-attractor}, at least for hard potentials.

Note that we stated Theorem \ref{t:global-attractor} for a solution in an interval of time $[0,T]$. The constants $C_p$ in the estimate are independent of this value of $T$. The result in \cite{desvillettes2005global} is stated for global solutions defined for all time. However, the estimates at time $t$ naturally do not depend on anything about the solution $f$ after that time. The way we intend to use Theorem \ref{t:global-attractor} is that if the solution $f$ exists in a time interval $[0,T]$, for large enough $T$, then $f$ will be very close to the Maxwellian $M$ at the final time $t=T$.

Thanks to the regularity estimates of Theorem \ref{t:reg} and the lower bounds established in \cite{imbert-mouhot-silvestre-lowerbound2020}, the conclusion of Theorem \ref{t:global-attractor} holds for any classical solution $f$ of \eqref{e:boltzmann} for which \eqref{e:hydro} holds.

\subsection{Carleman decomposition}\label{s:carleman}

Following \cite{silvestre-boltzmann2016}, we write the collision operator as the sum $Q(f,g) = Q_s(f,g) + Q_{ns}(f,g)$, with 
\begin{equation}\label{e:Qs}
 Q_{s}(f,g)(v) = \int_{\R^d} [g(v') - g(v)] K_f(v,v') \dd v',  
 \end{equation}
with 
\[ K_f(v,v') = \frac {2^{d-1}} {|v' - v|} \int_{\{w\perp (v'-v)\}} f(v+w) B(r,\cos\theta)r^{2-d} \dd w,\]
where $r^2 = |v'-v|^2 + |w|^2$ and $\cos(\theta/2) = |w|/r$. This kernel can be bounded by
\begin{equation}\label{e:Kf}
 K_f(v,v') \approx \left( \int_{\{w\perp (v'-v)\}} f(v+w)|w|^{\gamma+2s+1} \dd w\right) |v'-v|^{-d-2s},
\end{equation}
For the second term, one has
\begin{equation}\label{e:Qns}
 Q_{ns}(f,g)(v) = C (f\ast |\cdot|^\gamma)(v) g(v).
 \end{equation}
The constant $C$ and the implied constants in \eqref{e:Kf} depend only on $\gamma$, $s$, and $C_b$. The formulas \eqref{e:Qs} and \eqref{e:Qns} are valid for any two functions $f,g:\R^d\to \R$ such that the right-hand sides make sense. For proofs of these formulas, see Sections 4 and 5 of \cite{silvestre-boltzmann2016}. We also recall the following form of the cancellation lemma (see \cite[Lemma 3.6]{imbert-silvestre-whi2020}).
\begin{equation} \label{e:cancellation}
\int_{\R^d} [K_f(v,v') - K_f(v',v)] \dd v' = C (f\ast |\cdot|^\gamma)(v),
\end{equation}
with the same constant $C$ as in \eqref{e:Qns}.

\subsection{Miscellaneous estimates for the collision operator}

In this section we collect some estimates from the literature involving the collision operator $Q$. They will be used for the short-time existence result in Section \ref{s:short-time}.

The following is an elementary estimate for the upper bound of the convolution of a function $f$ and $|v|^\kappa$, for any value of $\kappa$. It is the same as \cite[Lemma 2.4]{imbert2019global}. In this paper, we will apply it for $\kappa = \gamma$ and for $\kappa = \gamma+2s$.

\begin{lemma} \label{l:convolution-C0}
Let $\kappa > -d$ and $f : \R^d \to [0,\infty)$ be a function such that $f(v) \leq N \vv^{-q}$ for some $q>d+\kappa_+$. Then
\[ \int_{\R^d} f(v+w) |w|^\kappa \dd w \leq C N \vv^\kappa.\]
Here the constant $C$ depends on $d$, $\kappa$ and $q$ only. Moreover, $C$ can be taken independent of $q$ provided that $q \geq q_0$ for some $q_0>d+\kappa_+$.
\end{lemma}

The final statement about the choice of $C$ independently of $q$ follows from the simple observation that if $f(v) \leq N \vv^{-q}$, then also $f(v) \leq N \vv^{-q_0}$.

The kernel $K_f$ of \eqref{e:Kf} satisfies some ellipticity bounds depending only on macroscopic quantities associated to $f$. When $f \geq 0$, we also get from \eqref{e:Kf} that $K_f \geq 0$. In some parts of this paper we will evaluate $Q(f,M)$ for a function $f$ that may change sign. Thus, the kernel $K_f$ may change sign as well. We observe that $|K_f(v,v')| \leq K_{|f|}(v,v')$, and this allows us to deduce some basic estimates for the kernel $K_f$ regardless of the sign of $f$.

The following lemma is proved in \cite{silvestre-boltzmann2016}.
\begin{lemma} \label{l:K_upper_bound}
Let $K_f$ be the kernel given by the formula \eqref{e:Kf}. Then, for any $r>0$ and $v \in \R^d$,
\[\begin{split}
 \int_{B_r(v)} |v'-v|^2 |K_f(v,v')| \dd v' &\leq C r^{2-2s} \int_{\R^d} |f(v-w)| |w|^{\gamma+2s} \dd w,\\
 \int_{\R^d\setminus B_r(v)} |K_f(v,v')| \dd v' &\leq C r^{-2s} \int_{\R^d} |f(v-w)| |w|^{\gamma+2s} \dd w.
 \end{split}
 \]
Here, the constant $C$ depends on dimension, $s$ and $\gamma$ only (not on $f$).
\end{lemma}

Bounds for the kernel $K_f$ such as the one in Lemma \ref{l:K_upper_bound} help us estimate the value of the application of the integral operator to a smooth function $\varphi$. The following lemma can be found in \cite[Lemma 4.6]{imbert2019global} and in \cite[Lemma 2.3]{imbert-mouhot-silvestre-lowerbound2020}.

\begin{lemma} \label{l:pointwise_operator_bound}
Let $K: \R^d \to \R$ be a symmetric kernel (i.e. $K(v,v+w) = K(v,v-w)$) so that 
\[ \int_{B_r(v)} |v'-v|^2 |K(v,v')| \dd v' \leq \Lambda r^{-2s}.\]
Consider the integro-differential operator $\mathcal L_K$,
\[ \mathcal L_K \varphi(v) = \PV \int_{\R^d} (\varphi(v') - \varphi(v)) K(v,v') \dd v'.\]
If $\varphi$ is bounded in $\R^d$ and $C^{\alpha}$ at $v$ for some $2s < \alpha \leq 2$, then
\[ \left| \mathcal L_K \varphi(v) \right| \leq C \Lambda |\varphi|_{C^0(\R^d)}^{1-\frac {2s} {\alpha}} [\varphi]_{C^{\alpha}(v)}^{\frac {2s} {\alpha}}.\]
The constant $C$ depends on dimension, $s$ and $\alpha$.
\end{lemma}

We use the (more or less standard) notation $[\varphi]_{C^{\alpha}(v)}$ for the $C^\alpha$ semi-norm localized at the point $v$
\[ [\varphi]_{C^{\alpha}(v)}:= \inf \left\{ \sup_{v' \in \R^d} \frac{|\varphi(v') - p(v')|}{|v'-v|^\alpha} : \text{for any polynomial } p \text{ so that } \deg p < \alpha \right\}. \]
In particular $[\varphi]_{C^{\alpha}(v)} \leq [\varphi]_{C^{\alpha}(\R^d)}$, and it depends on the point $v$. Observe for example that
\[ [ \langle \cdot \rangle^{-q} ]_{C^\alpha(v)} \approx q^\alpha \langle v \rangle^{-\alpha}. \]

Combining Lemma \ref{l:K_upper_bound} with Lemma \ref{l:pointwise_operator_bound}, we obtain the following.

\begin{cor}\label{c:Qsbound}
If $g$ is a bounded and $C^2$ function, then the operator $Q_s(f,g)$ defined by \eqref{e:Qs} satisfies
\[ |Q_s(f,g)| \lesssim \|g\|_{L^\infty(\R^d)}^{1-s} [g]_{C^2(v)}^s  \int_{\R^d} |f(v-w)||w|^{\gamma+2s} \dd w.\]
\end{cor}

\subsection{Coercivity and trilinear estimates}

We also require some estimates on the collision operator in Sobolev norms. A coercivity estimate for the Boltzmann collision operator with an optimal weight is given in \cite{gressman-strain2011sharp} in terms of the custom-defined seminorm $\dot N^{s,\gamma}$ given by
\begin{equation}\label{e:Ns-gamma}
 \|g\|_{\dot N^{s,\gamma}}^2 := \iint_{\R^d\times\R^d} \left( \vv \langle v'\rangle\right)^{(\gamma+2s+1)/2} \frac {(g(v') - g(v))^2}{d(v,v')^{d+2s}} \one_{\{d(v,v')\leq 1\}} \dd v',
\end{equation}
where 
\[ d(v,v') := \sqrt{ |v-v'|^2 + \frac 1 4 (|v|^2 - |v'|^2)^2}.\]
The coercivity estimate of \cite{gressman-strain2011sharp} is stated as follows:
\begin{prop}\label{p:coercive}
For $f$ satisfying the assumptions \eqref{e:hydro}, there exists a constant $c_0>0$ depending only on the constants $m_0$, $M_0$, $E_0$, and $H_0$, such that
\[ \int_{\R^d} Q(f, g) g \dd v \leq -c_0 \|g\|_{\dot N^{s,\gamma}(\R^d)}^2 + C_0 \|f\|_{L^1_{\gamma}(\R^d)} \|g\|_{L^2_{\gamma/2}(\R^d)}^2.\]
\end{prop}
We also follow \cite{gressman-strain2011sharp} in defining the $N^{s,\gamma}$ norm as
\[ \|g\|_{N^{s,\gamma}}^2 :=  \|g\|_{\dot N^{s,\gamma}}^2 + \|g\|_{L^2_{(\gamma+2s)/2}}^2.\]
Note that the definition of $L^2_q$ in \cite{gressman-strain2011sharp} is not the same as the one we use here. It corresponds to $L^2_{q/2}$ in our notation.

The following trilinear estimate is also proved in \cite{gressman-strain2011sharp}.

\begin{prop}\label{p:trilinear}
\[ 
\left|\int_{\R^d} Q(f,g) h \dd v\right| \lesssim \|f\|_{L^1_{(\gamma+2s+2)}} \|g\|_{N^{s,\gamma}} \|h\|_{N^{s,\gamma}}.\]
\end{prop}

The norm $\|f\|_{L^1_{(\gamma+2s+2)}(\R^d)}$ is an upper bound for the constant of Assumption U in \cite{gressman-strain2011sharp}. One can probably extract a smaller exponent than $\gamma+2s+2$ from carefully inspecting the proof in \cite{gressman-strain2011sharp}. In the coercive estimate of Proposition \ref{p:coercive}, it is easy to see that the factor $\|f\|_{L^1_\gamma}$ suffices, even though it is also stated in terms of Assumption U in \cite{gressman-strain2011sharp}.

The following estimate for the kernel $K_f$ is closely related to Propositions \ref{p:coercive} and \ref{p:trilinear}. A proof of the lower bound can be found explicitly in Appendix A of \cite{imbert-silvestre-whi2020}, and it would also follow from the analysis in \cite{gressman-strain2011sharp}. The upper bound follows from Proposition \ref{p:trilinear} and the expression \eqref{e:Qns} for $Q_{ns}$:

\begin{cor} \label{c:GS}
If $f$ satisfies the assumptions \eqref{e:hydro}, then 
\[ c_0 \|g\|_{\dot N^{s,\gamma}}^2 \leq \iint_{\R^d\times\R^d} |g(v') - g(v)|^2 K_f(v,v') \dd v' \dd v \leq C_1 \|f\|_{L^1_{\gamma+2s+2}} \|g\|_{N^{s,\gamma}}^2,\]
where $c_0>0$ depends on the mass, energy, and entropy bounds for $f$, as in Proposition \ref{p:coercive}.
\end{cor}

The following is a commutator estimate in terms of the $N^{s,\gamma}$ norms of \cite{gressman-strain2011sharp}. It is apparently new.

\begin{lemma} \label{l:commutator}
For any $q>0$, and $f,g : \R^d \to \R$, we have
\[ \| \langle v \rangle^q Q(f,g) - Q(f,\langle v \rangle^q g)\|_{L^2_{-\gamma_+/2}(\R^d)} \leq C_q \|f\|_{L^1_{\gamma+2s+2}} \|\langle v \rangle^q g\|_{N^{s,\gamma}}. \]
\end{lemma}

\begin{proof}
We expand the integral expression for $Q$ in terms of $K_f$ to get
\begin{align*}
\langle v \rangle^q Q(f,g)(v) - Q(f,\langle v \rangle^q g)(v) &= \int_{\R^d} (\langle v \rangle^q - \langle v' \rangle^q) g(v') K_f(v,v') \dd v' \\
&= \int_{\R^d} (\langle v \rangle^q \langle v' \rangle^{-q} - 1) \langle v' \rangle^q g(v') K_f(v,v') \dd v' \\
&= \int_{\R^d} (\langle v \rangle^q \langle v' \rangle^{-q} - 1) (\langle v' \rangle^q g(v') - \langle v \rangle^q g(v)) K_f(v,v') \dd v' \\
&\phantom{=} + \langle v \rangle^q g(v) \int_{\R^d} (\langle v \rangle^q \langle v' \rangle^{-q} - 1) K_f(v,v') \dd v' \\
& =: I(v) + II(v).
\end{align*}

For the first term, we use that $|K_f| \leq K_{|f|}$ and we apply Cauchy-Schwarz.
\begin{align*}
I(v) &:= \int_{\R^d} (\langle v \rangle^q \langle v' \rangle^{-q} - 1) (\langle v' \rangle^q g(v') - \langle v \rangle^q g(v)) K_f(v,v') \dd v' \\
&\leq \left( \int_{\R^d} (\langle v \rangle^q \langle v' \rangle^{-q} - 1)^2 K_{|f|} \dd v' \right)^{1/2} \cdot \left( \int_{\R^d} (\langle v' \rangle^q g(v') - \langle v \rangle^q g(v))^2 K_{|f|} \dd v' \right)^{1/2} .
\end{align*}

Using Corollary \ref{c:Qsbound} with $\varphi = \vv^{-q}$, we observe that
\begin{equation}\label{e:commutator1}
 \int_{\R^d} (\langle v \rangle^q \langle v' \rangle^{-q} - 1)^2 K_{|f|} \dd v' \leq C_q \langle v \rangle^{-2s} (f \ast |v|^{\gamma+2s}) \leq C \langle v \rangle^{\gamma_+} \|f\|_{L^1_{\gamma+2s}} .
 \end{equation}
Therefore
\[ \int_{\R^d} \langle v \rangle^{-\gamma_+} I(v)^2 \dd v \leq C \|f\|_{L^1_{\gamma+2s}} \iint (\langle v' \rangle^q g(v') - \langle v \rangle^q g(v))^2 K_{|f|} \dd v' \dd v.  \]
Using Corollary \ref{c:GS}, we conclude that
\[ \|I\|_{L^2_{-\gamma_+/2}}^2 \leq C \|f\|_{L^1_{\gamma+2s}} \|f\|_{L^1_{\gamma+2s+2}} \|g\|_{N^{s,\gamma}}^2.\]

This takes care of the first term $I(v)$. Let us now analyze the second one: $II(v)$. This is a lower order term. Indeed, a crude estimate similar to \eqref{e:commutator1} tells us that
\[ \left\vert \int_{\R^d} (\langle v \rangle^q \langle v' \rangle^{-q} - 1) K_f(v,v') \dd v' \right\vert \leq C \|f\|_{L^1_{\gamma+2s}} \langle v \rangle^{\gamma_+}. \]
We get the pointwise estimate $|II(v)| \leq  C \|f\|_{L^1_{\gamma+2s}} \langle v \rangle^{q+\gamma_+} g(v)$. From this, we get
\[ |II(v)|_{L^2_{\gamma_+/2}} \leq C \|f\|_{L^1_{\gamma+2s}} \|\langle v \rangle^q g\|_{L^2_{\gamma_+/2}}.\]
This takes care of the second term.
\end{proof}

Finally, we have a simple interpolation lemma that allows us to trade decay for regularity. The proof is the same as \cite[Lemma 2.6]{henderson2020polynomial}.
\begin{lemma}\label{l:moment_interpolation}
	Fix $n \geq 0$ and $m \geq 0$.  Suppose that $f \in L^\infty_m\cap H^k_n(\R^d)$ and $k' \in (0,k)$.  Then if $\ell < (m-d/2)(1- k'/k) + n(k'/k)$, we have
	\[
		\|f\|_{H^{k'}_\ell(\R^d)}
			\lesssim \|f\|_{L^\infty_m(\R^d)}^{1 - \frac{k'}{k}} \|f\|_{H^k_n(\R^d)}^{\frac{k'}{k}} \lesssim_{k',k} \|f\|_{L^\infty_m(\R^d)} + \|f\|_{H^k_n(\R^d)}.
	\]
\end{lemma}

\section{Short-time existence} 
\label{s:short-time}

We need a short-time existence theorem that allows initial data to decay only polynomially in $v$ (rather than exponential or Gaussian decay). This was established in \cite{morimoto2015polynomial} for $s\in (0,\frac 1 2)$ and $\gamma \in (-\frac 3 2, 0]$, and in \cite{henderson2020polynomial} for $s\in (0,1)$ and $\gamma \in (\max\{-3,-\frac 3 2 - 2s\},0)$, but these results do not apply to the case $\gamma >0$.
Here, we provide a relatively quick proof of short-time existence when the initial data is near a Maxwellian, that applies both for $\gamma\leq 0$ and $\gamma > 0$.

\begin{prop} \label{p:short_time}
There exists a sufficiently large $q > 0$,  depending on $\gamma$ and $s$,  so that for any $\eps>0$ and any $T>0$, there exists a $\delta>0$ so that if $|f_0| \leq \delta \langle v \rangle^{-q}$, then there exists some bounded classical solution $f : [0,T] \times \T^3 \times \R^3 \to \R$ to \eqref{e:boltzmann-perturbative} that agrees with the initial data $f_0$ in the sense described in Section \ref{s:weak_solutions}, and such that
\[ |f(t,x,v)| \leq \eps \langle v \rangle^{-q}, \text{ for all } t,x,v \in [0,T] \times \T^3 \times \R^3.\]

Moreover, for $\gamma \leq 0$, if $|f_0| \leq C_0 \langle v \rangle^{- \tilde q}$ for any $\tilde q > q$ and $C_0 > 0$, then also $|f(t,x,v)| \leq C_1 \langle v \rangle^{- \tilde q}$ for some $C_1$ depending on $C_0$, $\tilde q$ and the parameters of the equation only.

Also, if $\gamma > 0$, then $|f(t,x,v)| \leq C_p(t) \langle v \rangle^{-p}$ for all $p > 0$, for some function $C_p(t)$ depending on $p$ and the parameters of the equation only.

Even furthermore, if $f_0 \in H^k_{q_1}(\T^3\times\R^3) \cap L^\infty_{\tilde q}(\T^3\times\R^3)$ for some $k \geq 4$, $q_1>0$, and $\tilde q>0$ large depending on $q_1$, then also $f(t,\cdot,\cdot) \in H^k_{q_1}(\T^3\times \R^3)$ for all $t \in [0,T]$.
\end{prop}

We emphasize that the time $T$ of existence in Proposition \ref{p:short_time} depends on $\| \langle v \rangle^q f_0 \|_{L^\infty} < \delta$ only. It is unaffected by any of the other norms of $f_0$.

This section is devoted to the proof of Proposition \ref{p:short_time}. Like most proofs of short-time existence results, part of the proof consists in establishing certain estimates that persist for some period of time, depending on the size of the initial data. We also need to devise some procedure to build the solution. We write an approximate problem whose solution can be easily constructed using results from the literature. We prove that the solutions to this approximate problem satisfy our estimates uniformly and pass to the limit.

\subsection{The construction of an approximate problem}\label{s:approximate_problem}

Let us start by describing our approximate problem. First, we consider the case that $f_0$ is $C^\infty$ and decays faster than any polynomial as $|v|\to \infty$. Later, we will analyze the general case by a standard approximation procedure.

Let $h>0$ be a small parameter. We will ultimately take $h \to 0$ to construct a solution of \eqref{e:boltzmann-perturbative}. It is convenient to take $h = T/N$ for some large integer $N \in \mathbb N$. We construct an approximate solution $f^h$ of \eqref{e:boltzmann-perturbative} as follows.

We partition the interval $[0,T]$ in subintervals
\[ [0,T) = [t_0=0,t_1) \cup [t_1,t_2) \cup \dots [t_{N-1}, t_N = T).\]
We choose this partition so that $t_i - t_{i-1} = h$ for all $i =1,2,\dots,N$. 

We divide these subintervals in two sets:
\begin{align*}
\mathcal D := [0,t_1) \cup [t_2,t_3) \cup [t_4,t_5) \cup [t_6,t_7) \cup \cdots = \bigcup_{i \text{ odd}} [t_{i-1},t_i), \\
\mathcal T := [t_1,t_2) \cup [t_3,t_4) \cup [t_5,t_6) \cup [t_7,t_8) \cup\cdots = \bigcup_{i \text{ even}} [t_{i-1},t_i).
\end{align*}

In $\mathcal D$, we let $f^h$ solve the space homogeneous Boltzmann equation for each fixed value of $x$. More precisely,
\begin{equation} \label{e:space-homogeneous-part}
\partial_t f^h = 2 Q(M+f^h, M+f^h), \quad t\in [t_{i-1},t_i), i \text{ odd},
\end{equation}
with $f^h(0,x,v) = f_0(x,v)$

This equation is solvable thanks to the result in \cite{he2012homogeneous-boltzmann}, which we quote here:
\begin{thm}{\cite[Theorem 1.1(1)]{he2012homogeneous-boltzmann}}
If $\gamma+2s\geq 0$,  and $f_0 \in L^1_r \cap H^N(\R^3)$ for some $N\geq 3$ and $r>0$ sufficiently large depending only on $N$, then the space homogeneous Boltzmann equation
\[ \partial_t g = Q(g,g)\]
admits a global, unique solution satisfying $g \in C([0,T], H^N(\R^3)) \cap L^\infty([0,T], L^1_r \cap L\log L(\R^3))$ for any $T<\infty$.
\end{thm}
We apply this theorem with $N=3$, and our solution $f^h$ is obtained from $g$ by subtracting off $M(v)$ and rescaling time to address the factor of $2$ in \eqref{e:space-homogeneous-part}. Therefore, we require the initial data $f^h(t_{i-1},x,\cdot)$ to belong to the space $L^1_r \cap H^3(\R^3)$ for some $r$ sufficiently large, for each value of $x$. 

We assume that $f_0$ is smooth and rapidly decaying for large velocities. Thus, the assumptions to apply the result from \cite{he2012homogeneous-boltzmann} hold at $i=1$. We have to make sure that these assumptions will also hold for later values of $i$.

The solutions obtained by solving the space-homogeneous problem are smooth and decay in $v$ faster than any polynomial. Note also that $M+f^h \geq 0$ for all $h>0$. There is no easily applicable result that guarantees that the solution $f^h$ will be smooth with respect to $x$. Because of that, at each $t_i$ for $i\geq 1$ odd, we replace $f(t_i,x,v)$ with a mollification in space:
\begin{equation} \label{e:mollification}
f^h_i(x,v) := \int_{\R^3} f^h(t_i,x-y,v)h^{-3} \chi(y/h) \dd y,
\end{equation}
where $\chi$ is a smooth, even, nonnegative function supported in $B_1(0)$, with $\int_{B_1(0)} \chi = 1$. The functions $f_i^h$ will be smooth both in $x$ and $v$, and rapidly decaying as $|v| \to \infty$.

In each interval $[t_{i-1},t_i) \subset \mathcal T$, with $i$ even, we let $f^h$ solve the transport equation with initial data $f_{i-1}^h$ as in \eqref{e:mollification}. That is,
\[ \begin{cases} \partial_t f^h = -2 v\cdot \nabla_x f^h, & t\in [t_{i-1},t_i), \, i \text{ even}\\ f^h(t_{i-1},x,v) = f_{i-1}^h(x,v).\end{cases}\]

The transport equation preserves the smoothness of the function $f$ in $\mathcal T$. Thus, the initial data for solving problem \eqref{e:space-homogeneous-part} is always smooth in all variables and the result from \cite{he2012homogeneous-boltzmann} will be applicable for every subinterval in $\mathcal D$.

Note that by the mollification introduced in \eqref{e:mollification}, the function $f^h$ will have a jump discontinuity with respect to time at every $t_i$ with $i$ odd. The scaling of the mollification is chosen so that its quantitative impact on the solution disappears as $h\to 0$, as we demonstrate below. It is convenient for our analysis below to make $f^h$ be left continuous, that is $f^h(t_i,x,v) = \lim_{t \to t_i^-} f^h(t,x,v)$ for every $i$ odd.

\subsection{Uniform upper bounds}
\label{s:upper_bounds}

Now that we have constructed the approximate solution $f^h$ for every $h>0$, it is time to obtain some estimates that are independent of $h$. In this section, we obtain estimates in weighted $L^\infty$ following the ideas from \cite{imbert-mouhot-silvestre-decay2020}. 

In this subsection, since the analysis applies in arbitrary dimension, we prove our estimates for functions defined on $\R^d$ for general $d\geq 2$. The construction of $f^h$ in the previous subsection requires $d=3$ because of the application of \cite{he2012homogeneous-boltzmann}.

Let $g : \R^d \to (0,2]$ be a smooth function satifying 
\begin{equation}\label{e:g-def}
g(v) = |v|^{-q} \, \text{ for } |v| > 1, \qquad g(v) \geq \vv^{-q} \, \text{ in } \R^d.
\end{equation}
We use this definition for $g$ during this whole subsection. Obviously $g(v) \approx \langle v \rangle^{-q}$. We use the function $g(v)$ instead of $\langle v \rangle^{-q}$ because we want to use some computations from \cite{imbert-mouhot-silvestre-decay2020} that apply to those functions. Choosing $g$ such that $g(v) \geq \vv^{-q}$ is done for technical convenience.

We start with a few lemmas that involve various upper bounds for different parts of the Boltzmann collision operator. The first one  involves the quantity $Q_s(M+f^h, g)$, evaluated at some given point $(\bar t, \bar x, \bar v)$. Following \cite{imbert-mouhot-silvestre-decay2020}, we split this quantity as $Q_s(M+f^h, g) = \mathcal G + \mathcal B$. For $c_1(q) = q^{-1}/20$, we define
\begin{align}
\mathcal G &:= Q_s(\one_{\langle v \rangle < c_1(q)|\bar v|}(M+f^h),g)(\bar v), \label{e:G}\\ 
\mathcal B &:= Q_s(\one_{\langle v \rangle \geq c_1(q)|\bar v|} (M+f^h),g)(\bar v), \label{e:B}
\end{align}

The point of this decomposition is that, as we will see, $\mathcal G < 0$, and $\mathcal B$ is an error term that must be estimated from above.

Note that for any value of $w \in \R^d$, $\langle w \rangle \geq 1$. Thus, we have $Q_s(M+f^h, g)(\bar v) = \mathcal B$ when $|\bar v|$ is sufficiently small. The decomposition is only useful to study large values of $|\bar v|$.

\begin{lemma}\label{l:good-term}
Assume $|f^h(v)| \leq g(v)/2$ for all $v\in\R^d$, and $M+f^h \geq 0$. There exists $R_q>0$ depending on $q$, such that the term $\mathcal G$ defined by \eqref{e:G} satisfies
\[\mathcal G \leq -c (1+q)^s |\bar v|^\gamma g(\bar v),\]
if $|\bar v|>R_q$. The constant $c>0$ is independent of $q$.
\end{lemma}

\begin{proof}
By Lemma \ref{l:hydro-near_equilibrium}, the bound $M-g(v)/2 \leq (M+f^h) \leq M+g(v)/2$ implies the bounds \eqref{e:hydro} on the hydrodynamic quantities on the function $M+f^h$. Thus, this lemma is the same as \cite[Proposition 3.1]{imbert-mouhot-silvestre-decay2020}.
\end{proof}

The term $\mathcal B$ may be estimated by the combination of Propositions 3.7, 3.8, and 3.9 in \cite{imbert-mouhot-silvestre-decay2020}. However, because the statements of these propositions are not completely explicit for our purposes, and the setting in this paper is slightly simpler, we present the computation to estimate $\mathcal B$ in full detail. In order to do so, we further split the term $\mathcal B = \mathcal B_1 + \mathcal B_2$, where
\begin{align*}
\mathcal B_1 &:= Q_s(\one_{\langle v \rangle \geq c_1(q)|\bar v|} (M+f^h), \one_{\langle v \rangle < |\bar v|/2} g)(\bar v), \\
\mathcal B_2 &:= Q_s(\one_{\langle v \rangle \geq c_1(q)|\bar v|} (M+f^h), \one_{\langle v \rangle > |\bar v|/2} g)(\bar v).
\end{align*}

The following is an auxiliary lemma.
\begin{lemma} \label{l:b2}
Assuming $q > d+\gamma+2s$ and $|f^h(v)| \leq U g(v)$ for all $v\in \R^d$, for any $r\geq 1$ and $v \in \R^d$ there holds
\[ \int_{\langle w \rangle \geq r} |f^h(w)| |v-w|^{2s+\gamma} \dd w \lesssim \frac{U}{q-d} r^{-q+d} |v|^{\gamma+2s} + \frac{U}{q-d-\gamma-2s} r^{-q+d+\gamma+2s}. \]
\end{lemma}

\begin{proof}
By a direct computation, using that $|v-w|^{2s+\gamma} \lesssim |v|^{2s+\gamma} + |w|^{2s+\gamma}$,
\begin{align*} 
\int_{\langle w \rangle \geq r} |f^h(w)| |v-w|^{2s+\gamma} \dd w  &\lesssim U \int_{\langle w \rangle \geq r} g(w) ( |v|^{2s+\gamma} + |w|^{2s+\gamma} ) \dd w, \\
&\lesssim U \left( \frac{1}{q-d} r^{-q+d} |v|^{2s+\gamma} + \frac{1}{q-d-\gamma-2s} r^{-q+d+\gamma+2s} \right).
\end{align*}
\end{proof}

We now estimate $\mathcal B_2$.

\begin{lemma} \label{l:B2}
Assuming $q > d+\gamma+2s$ and $|f^h(v)| \leq U g(v)$ for all $v\in \R^d$, there is a constant $C_q$ so that
\[ \mathcal B_2 \leq C_q (1+U) \langle \bar v \rangle^{-2q + d + \gamma}.\]
\end{lemma}

\begin{proof}
Let $g_{\bar v}(v) := \one_{\vv > |\bar v|/2} g(v)$. Applying Corollary \ref{c:Qsbound}, we get
\[ \mathcal B_2 \lesssim \|g_{\bar v}\|_{L^\infty}^{1-s} [g_{\bar v}]_{C^2(\bar v)}^{s} \int_{\langle w \rangle \geq c_1(q)|\bar v|} (M(w)+f^h(w)) |\bar v-w|^{2s+\gamma} \dd w. \]

By a direct computation, we verify that
\[ \|g_{\bar v}\|_{L^\infty} \approx \langle \bar v \rangle^{-q}, \qquad [g_{\bar v}]_{C^2} \approx q^2 \langle\bar v\rangle^{-q-2}.\]
Therefore,
\begin{equation}\label{e:b1}
\mathcal B_2 \lesssim q^{2s} \langle\bar v\rangle^{-q-2s} \int_{\langle w \rangle \geq c_1(q)|\bar v|} (M(w)+f^h(w)) |\bar v-w|^{2s+\gamma} \dd w.
\end{equation}

In order to estimate the integral on the right hand side, we compute each of its two terms separately. On one hand, assuming $q>d+\gamma+2s$, Lemma \ref{l:b2} tells us that
\begin{align*} 
\int_{\langle w \rangle \geq c_1(q)|\bar v|} |f^h(w)| |\bar v-w|^{2s+\gamma} \dd w  &\lesssim C_q U \langle \bar v \rangle^{-q+d+2s+\gamma}
\end{align*}
The term involving $M(w)$ decays faster than exponential for large velocities. In order to do this proof, we only need the following gross overestimation
\begin{align*} 
\int_{\langle w \rangle \geq c_1(q)|\bar v|} M(w) |\bar v-w|^{2s+\gamma} \dd w  &\lesssim C_q \langle \bar v \rangle^{-q+d+2s+\gamma}
\end{align*}

Replacing the last two bounds back into \eqref{e:b1}, we complete the proof.
\end{proof}

We now proceed to estimate the term $\mathcal B_1$
\begin{lemma} \label{l:B1}
Assuming $q > d+\gamma+2s$ and $|f^h(v)| \leq U g(v)$ for all $v\in \R^d$, we have
\[ \mathcal B_1 \lesssim \frac{U}{q-d-\gamma-2s} \langle \bar v \rangle^{-q + \gamma} +  \exp(-\langle \bar v \rangle).\]
\end{lemma}

\begin{proof}
Note that for $\langle v' \rangle < |\bar v|/2$, we have $g(v') - g(\bar v) \geq 0$. Recalling the integral expression for $\mathcal B_1$, we have
\[ \mathcal B_1 \lesssim \int_{\langle v' \rangle < |\bar v|/2} \int_{\{w\perp(v'-\bar v), \langle \bar v +w \rangle \geq c_1(q)|\bar v|\}} (M+f^h)(\bar v+w)|w|^{\gamma+2s+1} [g(v') - g(\bar v)] |v'-\bar v|^{-d-2s} \dd w \dd v', \]
with implied constant depending only on $C_b$ in \eqref{e:collision-kernel-assumptions}. The point that we want to make here is that given $\langle v' \rangle < |\bar v|/2$ and $(v'-\bar v) \perp w$, then we automatically have $|v+w| > \sqrt{3}/2 |\bar v|$ by a straightforward geometric argument. Thus, the expression bounding $\mathcal B_1$ can be simplified to
\begin{align*}
\mathcal B_1 &\lesssim \int_{\langle v' \rangle < |\bar v|/2} \int_{\{w\perp(v'-\bar v), \langle \bar v +w \rangle \geq \sqrt{3}/2|\bar v|\}} (M+f^h)(\bar v+w)|w|^{\gamma+2s+1} [g(v') - g(\bar v)] |v'-\bar v|^{-d-2s} \dd w \dd v', \\
\intertext{Switching the order of integration,}
& = \int_{\langle \bar v +w \rangle \geq \sqrt{3}/2|\bar v|} \int_{\{ (v'-\bar v) \perp w, \langle v' \rangle < |\bar v|/2 \} }  (M+f^h)(\bar v+w)|w|^{\gamma+2s} [g(v') - g(\bar v)] |v'-\bar v|^{-d-2s+1} \dd w \dd v', \\
\intertext{Using here $0\leq g(v') - g(\bar v) \lesssim 1$,}
& \lesssim \langle \bar v \rangle^{-2s} \int_{\langle \bar v +w \rangle \geq \sqrt{3}/2|\bar v|} (M+f^h)(\bar v+w)|w|^{\gamma+2s} \dd w,
\intertext{Applying Lemma \ref{l:b2} for the term involving $f^h$ and using that $M$ decays faster than any exponential,}
& \lesssim \frac{U}{q-d-\gamma-2s} \langle \bar v \rangle^{d+\gamma} +  \exp(-\langle \bar v \rangle)
\end{align*}
\end{proof}

Combining Lemmas \ref{l:B1} and \ref{l:B2}, we conclude the following upper bound for $\mathcal B$:

\begin{prop}\label{p:bad-term}
Assuming $q > d+\gamma+2s$ and $|f^h(v)| \leq U g(v)$ for all $v\in \R^d$, then
\[ \mathcal B \lesssim \frac{U}{q-d-\gamma-2s} \langle \bar v \rangle^{-q+\gamma} + C_q \left(1+U\right) \langle \bar v \rangle^{-2q+d+\gamma}.\]
\end{prop}

\begin{cor} \label{c:Q(M+f,g)}
Assume $|f^h(v)| \leq g(v)/2$ for all $v\in \R^d$. There exists a $q_0$ sufficiently large so that for all $q>q_0$ we have
\[ Q_s(M+f^h,g)(\bar v)  \leq (a_q - b_q \langle \bar v \rangle^{\gamma_+}) g(\bar v), \]
Moreover, $b_q \approx q^s$ for large $q$.
\end{cor}

\begin{proof}
Combine Lemma \ref{l:good-term} with Proposition \ref{p:bad-term}.
\end{proof}


We move on to analyze the term $Q_s(f^h,M)$ with the following Lemma.

\begin{lemma}\label{l:Qsfh}
Assume $|f^h(v)| \leq U g(v)$ for all $v\in \R^d$. There exists a universal constant $R>0$ such that if $|\bar v|>R$, then 
\[ |Q_s(f^h,M)(\bar v)| \lesssim U |\bar v|^\gamma g(\bar v),\]
with implied constant independent of $q$.
\end{lemma}

\begin{proof}
We have from \eqref{e:Qs} that $Q_s(f^h,M) = \int_{\R^d} K_{f^h}(\bar v,v') [M(v') - M(\bar v)] \dd v'$. We divide this integral into the parts where $|v'|\geq |\bar v|/2$ and $|v'|< |\bar v|/2$. For the first part, define
\[ \bar M(v) := \begin{cases} M(v), &|v|\geq |\bar v|/2,\\ M(\bar v), & |v|< |\bar v|/2.\end{cases}\]
We then have
\begin{align}
 \int_{\{|v'|\geq |\bar v|/2\}} K_{f^h}(\bar v,v') [M(v') - M(\bar v)] \dd v' &= \int_{\R^d} K_{f^h}(\bar v,v') [\bar M(v') - \bar M(\bar v)] \dd v' \nonumber \\
&\leq \|\bar M\|_{L^\infty(\R^d)}^{1-s} \|\bar M\|_{C^2_v(\R^d)}^s  \int_{\R^d} |f^h(\bar t,\bar x,w)| |\bar v-w|^{\gamma+2s} \dd w \nonumber \\
&\lesssim ( e^{-|\bar v|^2/8})^{1-s} ( |\bar v|^2 e^{-|\bar v|^2/8})^s U \langle \bar v\rangle^{\gamma+2s} \lesssim U \langle \bar v \rangle^{-q+\gamma}. \label{e:qM1}
\end{align}
The last inequality is a rather brutal bound from above. In the second line, we used Corollary \ref{c:Qsbound}, and in the third line, we used $|f^h(v)| \lesssim U g(v)$. 

Next, we address the integral over $\{v' : |v'|< |\bar v|/2\}$. Note that $M(\bar v)\leq M(v')$ in this region. Using \eqref{e:Kf} and the bound $|f^h(v)| \leq U g(v)$, we have
\[
\begin{split}
\int_{\{|v'|< |\bar v|/2\}} &K_{f^h}(\bar v,v') [M(v') - M(\bar v)] \dd v'\\ 
 &\lesssim  \int_{\{|v'|< |\bar v|/2\}} |v'-\bar v|^{-d-2s} M(v') \int_{\{w\perp (v'-\bar v) \}} |f^h(\bar v+w)||w|^{\gamma+2s+1} \dd w\dd v'\\
&\lesssim  |\bar v|^{-d-2s} \int_{\{|v'|< |\bar v|/2\}} M(v') \int_{\{w\perp (v'-\bar v)\}} U \langle \bar v + w\rangle^{-q} |w|^{\gamma+2s+1} \dd w \dd v'.
\end{split}
\]

When $|v'|< |\bar v|/2$, the hyperplane $\{\bar v + w : w\perp (v'-\bar v)\}$ is at distance at least $\sqrt 3/2 |\bar v|$
 from the origin. Because of that, we verify by a direct computation that
 \[ \int_{\{w\perp (v'-\bar v)\}} \langle \bar v + w\rangle^{-q} |w|^{\gamma+2s+1} \dd w \lesssim |\bar v|^{-q+\gamma+2s+d}.\]

Therefore
\begin{align}
\int_{\{|v'|< |\bar v|/2\}} &K_{f^h}(\bar v,v') [M(v') - M(\bar v)] \dd v'
\lesssim  U \langle \bar v \rangle^{-q+\gamma} \int_{\{|v'|< |\bar v|/2\}} M(v') \dd v' \lesssim U \langle \bar v \rangle^{-q+\gamma}. \label{e:qM2}
\end{align}

The proof is concluded combining \eqref{e:qM1} with \eqref{e:qM2}.
\end{proof}

Finally, we can state and prove the propagation of the upper bound for the approximate solution $f^h$ constructed in Section \ref{s:approximate_problem}.

\begin{lemma}\label{l:q-upper-bound}
Let $g : \R^d \to (0,2]$ be as in \eqref{e:g-def}. There exists a constant $C_1>0$ so that the following holds. If $|f_0(x,v)| \leq \delta g(v)$, then the approximate solution $f^h$ constructed in Section \ref{s:approximate_problem} satisfies $|f^h(t,x,v)| < \delta e^{C_1 t} g(v)$ uniformly in $h$, for any $t,x,v \in [0,T] \times \T^d \times \R^d$ provided that $\delta e^{C_1 T} < 1/2$. 
\end{lemma}

\begin{proof}
Let us write $U(t) := \delta e^{C_1 t}$, where the constant $C_1$ will be determined later. Naturally, $g(v) \approx \langle v \rangle^{-q}$. We should think of the function $U(t) g(v)$ as an upper barrier. We verify the inequality $f^h < U(t) g(v)$ by contradiction, using the classical idea of evaluating the equation at the first crossing point.

By construction, we have $|f^h(0,x,v)| = |f_0(x,v)| < \delta g(v) = U(0) g(v)$. We want to propagate this inequality to future values of $t$. The way we organize this proof is by showing that if $|f^h(t_i,x,v)| < U(t_i) g(v)$ for some $i \in \N$, then we also have $|f^h(t,x,v)| < U(t) g(v)$ for the whole interval $t \in [t_i,t_{i+1}]$. Naturally, the reason is different depending on whether $[t_i,t_{i+1}]$ is an interval in $\mathcal D$ or in $\mathcal T$.

Note that if $f^h$ satisfies $|f^h(t_i,x,v)| < U(t_i) g(v)$ for some $i$ odd, then the function $f_i$ given by the $x$-mollification \eqref{e:mollification} will trivially satisfy the same inequality. Moreover, in the interval $[t_i,t_{i+1}]$ $f^h$ solves the transport equation that merely moves values around in space. Any upper bound independent of $x$ will be invariant during every interval contained in $\mathcal T$. Since $U$ is increasing in time, then the inequality $|f^h(t_i,x,v)| < U(t_i) g(v)$ easily implies $|f^h(t,x,v)| < U(t_i) g(v) < U(t) g(v)$ for any $t,x,v \in (t_i,t_{i+1}] \times \T^d \times \R^d$.

We are left to analyze the case that $[t_i,t_{i+1}]$ is an interval in $\mathcal D$. In these intervals, the function $f^h$ solves the space-homogeneous problem \eqref{e:space-homogeneous-part} for every fixed value of $x$. Our analysis below is essentially an upper bound for the space-homogeneous Boltzmann equation.

Let us fix any point $\bar x \in \T^d$. We know that $|f^h(t_i,\bar x,v)| < U(t_i) g(v)$ and want to prove that $|f^h(t,\bar x,v)| < U(t) g(v)$ for every $t \in [t_i,t_{i+1}]$. For the sake of contradiction, let us suppose that this inequality is invalid somewhere. For this fixed value of $\bar x$, let $\bar t$ be the first time that the inquality does not hold:
\[ \bar t := \inf \left\{ t \in [t_i,t_{i+1}] : \text{ there exists } v \in \R^d \text{ such that } |f^h(t,\bar x,v)| \geq U(t) g(v) \right\} .\]

Since $f^h(\cdot,\bar x, \cdot)$ is continuous in $[t_i,t_{i+1}] \times \R^d$, and moreover $\lim_{|v| \to \infty} |f^h(t,\bar x,v)|/g(v) = 0$ uniformly for $t \in [t_i,t_{i+1}]$, then there must exist a first crossing point $\bar t , \bar v \in (t_i,t_{i+1}] \times \R^d$ so that
\begin{equation}\label{e:crossing-point-def}
\begin{split}
f^h(\bar t, \bar x, v) &= U\left(\bar t\right) g(\bar v), \\
f^h(t, \bar x, v) &< U(t) g(v) \text{ for any } t \in [t_i,\bar t), v \in \R^d.
\end{split}
\end{equation} 
Our plan is to obtain a contradiction by verifying that the equation \eqref{e:space-homogeneous-part} cannot hold at the point $(\bar t, \bar x, \bar v)$.

As it is standard for this type of barrier argument, we obtain a series of inequalities at the point $(\bar t, \bar x, \bar v)$. In this case,
\begin{equation}\label{e:crossing}
\begin{split}
\partial_t f^h(\bar t,\bar x,\bar v) &\geq U'(\bar t) g(\bar v), \\
Q(M+f^h,f^h)(\bar t, \bar x, \bar v) &\leq U\left(\bar t\right) Q(M+f^h(\bar t, \bar x, \cdot),g)(\bar v).
\end{split}
\end{equation}
The second of these inequalities is the standard ellipticity of the integral operator $Q_s$. It follows directly from the expression \eqref{e:Qs} observing that $K_{M+f^h}$ is a nonnegative kernel. Moreover, observe that $Q_{ns}(\varphi,f^h)(v) = Q_{ns}(\varphi,g)(v)$, for any functions $\varphi,f,g$, at any point where $f(v)=g(v)$.

From the equation \eqref{e:space-homogeneous-part} and the fact that $Q(M,M) = 0$, we have
\[ \partial_t f^h = Q(M+f^h,M+f^h) = Q(M+f^h,f^h) + Q(f^h,M). \]
Let us decompose the right hand side even further according to the decomposition of Section \ref{s:carleman}. Using \eqref{e:Qns} to write terms with $Q_{ns}$ as convolutions, we have
\begin{equation} \label{e:ste_decomp}
\partial_t f^h = C[(M+f^h)\ast|\cdot|^\gamma] \, f^h + C[f^h \ast|\cdot|^\gamma] \, M + Q_s( M+f^h, f^h) + Q_s(f^h,M).
\end{equation}
Evaluating this at the first crossing point $(\bar t, \bar x,\bar v)$ and using the second inequality in \eqref{e:crossing}, this gives
\begin{equation}\label{e:contact-point}
\partial_t f^h(\bar t, \bar x, \bar v) \leq U\left(\bar t\right) \bigg\{ [(M+f^h)\ast|\cdot|^\gamma] \, g(\bar v) + [g \ast|\cdot|^\gamma] \, M(\bar v) + Q_s(M+f^h, g)(\bar v) + Q_s(f^h/U\left(\bar t\right), M)(\bar v)  \bigg\}.
\end{equation}
Our goal is to bound this right-hand side from above and derive a contradiction with the first inequality in \eqref{e:crossing}. We use the lemmas proved earlier in this section to bound some of the terms. Moreover, we will see that for large enough $\bar v$, the term $Q_s(M+f^h, g)(\bar v)$ is strictly negative and it dominates all the other terms.

For the convolution terms in \eqref{e:contact-point}, we apply Lemma \ref{l:convolution-C0} to obtain 
\begin{equation}\label{e:convolution-terms}
[(M+f^h)\ast |\cdot|^\gamma](\bar v) g(\bar v) \lesssim |\bar v|^\gamma g(\bar v), \quad [g\ast|\cdot|^\gamma](\bar v) M(\bar v) \lesssim |\bar v|^\gamma M(\bar v),
\end{equation}
with constants independent of $q$. Note that $M$ decays much faster than $g\approx \langle v\rangle^{-q}$.

Combining Corollary \ref{c:Q(M+f,g)}, equation \eqref{e:convolution-terms}, and Lemma \ref{l:Qsfh}, we bound the right-hand side of \eqref{e:contact-point} and obtain
\[\partial_t f^h(\bar t,\bar x,\bar v) \lesssim  U(\bar t)[-cq^s |\bar v|^\gamma +2^q|\bar v|^{\gamma-2} + |\bar v|^\gamma] g(\bar v)  , \quad |\bar v|\geq R_q.\]
Picking $q$ sufficiently large, the negative term dominates, and we ensure that $\partial_t f^h(\bar t,\bar x,\bar v) < 0$ if $|\bar v|$ is large enough. Since our function $U(t)$ is increasing, the first contact point cannot occur with $|\bar v| > R_q$, for some large radius $R_q$ depending on our (finite) choice of $q$.

For $|\bar v| < R_q$, we estimate the right hand side in \eqref{e:contact-point} without any regard for the asymptotic behavior as $|\bar v|\to \infty$, and we get $\partial_t f^h < C_1 U\left(\bar t\right)$ for some universal constant $C_1$. This is the constant $C_1$ that we choose in the definition of $U(t) = \delta \exp(C_1 t)$ so that we obtain a contradiction with the first inequality in \eqref{e:crossing}.

We have shown $f^h$ is bounded above by $\delta \exp(C_1t) g(v)$ whenever the initial data is bounded by $\delta g(v)$ and provided that $U(t) \leq 1/2$ for all $t \in [0,T]$. We pick $\delta>0$ small enough so that $\delta \exp(C_1 T) < \min(\eps,1/2)$.
\end{proof}

In order to obtain the upper bounds with polynomial decay for higher powers, as in the statement of Proposition \ref{p:short_time}, we observe that the analysis in \cite{imbert-mouhot-silvestre-decay2020} applies directly to $M+f^h$ uniformly in $h$. These decay estimates depend only on the constants in \eqref{e:hydro}.

\subsection{Propagation of regularity}

In this section, we show that sufficiently high weighted Sobolev norms in the intial data $f_0$ are propagated forward to a positive time interval. The methods in this section are mostly classical since they involve coercivity and trilinear estimates in weighted Sobolev spaces. Yet, there are some delicate details that make the computations more cumbersome than one would expect.

In our proof of Proposition \ref{p:short_time}, we use the construction described in Section \ref{s:approximate_problem} with a smooth initial data. We later approximate any initial data $f_0 \in L^\infty_q$ with smooth functions, construct a solution for each of them, and pass to the limit. The estimates in this section are used for solving the approximate problem, but they will not apply to the final solution $f$ unless the initial data $f_0$ is smooth. The fact that we need these estimates for our construction might be an artifact of our proof of Proposition \ref{p:short_time}.

We begin with two auxiliary lemmas.

The following coercivity estimate gives us two negative terms. One corresponds to the usual Dirichlet form involving $K_f$. The other one is a weighted $L^2$ norm with a power $\gamma_+ + 2q$. The fact that this power is strictly larger than $2q$ when $\gamma>0$ can be seen as the reason why there is creation of $L^2$-moments in the hard potentials case.

\begin{lemma} \label{l:magic_negative_term_2}
There is a $q_0>0$ so that for any $q \geq q_0$, $f \geq 0$ satisfying $\|f-M\|_{L^\infty_{2q}} \leq \frac 1 2$, and any $\varphi : \R^d \to \R$, we have
\begin{align*} \int_{\R^d} Q(f,\varphi) \langle v \rangle^{2q} \varphi(v) \dd v \leq
-c_q \left( \iint_{\R^d\times\R^d} | \langle v' \rangle^q \varphi'- \langle v \rangle^q  \varphi|^2 K_f \dd v' \mathrm{d} v\right) +  \int_{\R^d} (a_q - b_q \langle v \rangle^{\gamma_+}) \langle v \rangle^{2q} \varphi(v)^2 \dd v.
\end{align*}
Here, $q_0$, $a_q \geq 0$, $b_q > 0$ and $c_q>0$ depend on $q$, dimension and the parameters on \eqref{e:hydro}. Moreover, $b_q \approx q^s$.
\end{lemma}

\begin{proof}
We expand the left hand side using \eqref{e:Qs} and \eqref{e:Qns}
\begin{align*}
\int_{\R^d} Q(f,\varphi) \varphi(v) \langle v \rangle^{2q} \dd v &= \int_{\R^d} C (f \ast |\cdot|^\gamma) \varphi(v)^2 \langle v \rangle^{2q} \dd v + \iint_{\R^d \times \R^d} K_f(v,v') (\varphi'-\varphi) \langle v \rangle^{2q} \varphi  \dd v' \dd v \\
&= I + \iint_{\R^d \times \R^d} K_f(v,v') (\langle v' \rangle^q \varphi'- \langle v \rangle^q \varphi) \langle v \rangle^{q} \varphi \dd v' \dd v \\
& \phantom{= I} + \iint_{\R^d \times \R^d} K_f(v,v') (\langle v \rangle^q \langle v' \rangle^{-q} -1)  \langle v' \rangle^{q} \varphi'  \langle v \rangle^{q} \varphi \dd v' \dd v.
\end{align*}
Here, we write $\varphi'$ and $\varphi$ to denote $\varphi(v')$ and $\varphi(v)$ respectively. Also, we wrote $I$ to denote the first term
\[ I := \int_{\R^d} C (f \ast |\cdot|^\gamma) \varphi(v)^2 \langle v \rangle^{2q} \dd v. \]
We continue the computation using \eqref{e:cancellation} and the change of variables $v \leftrightarrow v'$, and get 
\begin{align*}
\int_{\R^d} Q(f,\varphi) \varphi(v) \langle v \rangle^{2q} \dd v &= \frac 1 2 I - \frac 12 \iint_{\R^d \times \R^d} (\langle v' \rangle^{q} \varphi' - \langle v \rangle^{q} \varphi)^2 K_f(v,v') \dd v' \dd v  \\
& \phantom{= \frac 1 2 I} + \iint_{\R^d \times \R^d} K_f(v,v') (\langle v \rangle^q \langle v' \rangle^{-q} -1)  \langle v' \rangle^{q} \varphi'  \langle v \rangle^{q} \varphi \dd v' \dd v \\
&= \frac 1 2  I - \frac 12 \iint_{\R^d \times \R^d} (\langle v' \rangle^{q} \varphi' - \langle v \rangle^{q} \varphi)^2 K_f(v,v') \dd v' \dd v  \\
& \phantom{= \frac 1 2 I} + \iint_{\R^d \times \R^d} K_f(v,v') (\langle v \rangle^q \langle v' \rangle^{-q} -1)  (\langle v' \rangle^{q} \varphi' - \langle v \rangle^{q} \varphi ) \langle v \rangle^{q} \varphi \dd v' \dd v \\
& \phantom{= \frac 1 2 I} + \iint_{\R^d \times \R^d} K_f(v,v') (\langle v \rangle^q \langle v' \rangle^{-q} -1)  \langle v \rangle^{2q} \varphi^2 \dd v' \dd v.
\end{align*}
Let us pick $\eps\in (0,1)$ small enough, depending on $q$, and use that 
\[ (\langle v \rangle^q \langle v' \rangle^{-q} -1)  (\langle v' \rangle^{q} \varphi' - \langle v \rangle^{q} \varphi ) \langle v \rangle^{q} \varphi \leq \frac{1+\eps}2 (\langle v \rangle^q \langle v' \rangle^{-q} -1)^2  \langle v \rangle^{2q} \varphi^2 + \frac 1 {2(1+\eps)} (\langle v' \rangle^{q} \varphi' - \langle v \rangle^{q} \varphi )^2 .\]
Therefore,
\begin{equation}\label{e:coercivity1}
\begin{split}
\int_{\R^d} Q(f,\varphi) \varphi(v) \langle v \rangle^{2q} \dd v &\leq \frac 1 2 I - \frac \eps {2(1+\eps)} \iint_{\R^d \times \R^d} (\langle v' \rangle^{q} \varphi' - \langle v \rangle^{q} \varphi)^2 K_f(v,v') \dd v' \dd v  \\
& \phantom{= \frac 1 2 I} + \iint_{\R^d \times \R^d} K_f(v,v') \langle v \rangle^{2q} \varphi^2 \left\{ \frac 12 (\langle v \rangle^{2q} \langle v' \rangle^{-2q}-1) + \frac \eps 2 (\langle v \rangle^q \langle v' \rangle^{-q} -1)^2 \right\} \dd v' \dd v.
\end{split}
\end{equation}
For the last integral on the right, we first write
\begin{equation}\label{e:good-term1}
 \iint_{\R^d\times\R^d} K_f(v,v') \vv^{2q} \varphi^2 (\vv^{2q}\langle v'\rangle^{-2q} - 1)\dd v' \dd v=  \int_{\R^d} \vv^{4q} \varphi^2  Q_s(f, \vv^{-2q}) \dd v.
 \end{equation}
Let $g:\R^d\to \R$ be a smooth function which equals $|v|^{-2q}$ when $|v|>1$, as in \eqref{e:g-def} with exponent $2q$. We have a good upper bound for $Q_s(f,g)$ from Corollary \ref{c:Q(M+f,g)}, so we add and subtract $g$ inside $Q_s$:
\[ Q_s(f,\langle \cdot \rangle^{-2q})(v) = Q_s(f,g)(v) + Q_s(f,\langle \cdot \rangle^{-2q} - g)(v) \leq (a_q - b_q\vv^{\gamma_+}) \vv^{-2q} + Q_s(f,\langle \cdot \rangle^{-2q} - g)(v),\]
with $b_q \approx q^s$. We used Corollary \ref{c:Q(M+f,g)}\footnote{In the notation of Corollary \ref{c:Q(M+f,g)}, the current $f$ is denoted $M+f$.} and $g(v) \lesssim \vv^{-2q}$ in the last inequality. Next, we must estimate the resulting remainder term. Define $G(v) := \vv^{-2q} - g(v)$. If $|v|\leq 2$, then $Q_s(f,G)(v)$ is bounded by a constant depending only on $q$ and the constants in \eqref{e:hydro} (for example, by Corollary \ref{c:Qsbound}). Therefore, we assume $|v|> 2$ for our estimate of $Q_s(f,G)(v)$. We have
\begin{equation}\label{e:QsG}
\begin{split}
Q_s(f,\langle \cdot \rangle^{-2q} - g) &= 
 \left(\int_{B_{|v|/2}(v)}  + \int_{\R^d \setminus B_{|v|/2}(v)}\right) K_f(v,v') [ G' - G]\dd v'.
\end{split}
\end{equation}
For the first integral, we Taylor expand $G$ at $v$ and use the symmetry property $K_f(v,v') = K_f(v,v - (v'-v))$ to see that the first-order term integrates to zero. A straightforward calculation using $|\vv - |v|| \lesssim \vv^{-1}$ and $|\vv^{-p} - |v|^{-p}| \lesssim p \vv^{-p-2}$ shows that
\[\|D_v^2 G\|_{L^\infty(B_{|v|/2}(v))} = \| D_v^2 \vv^{-2q} - D_v^2 |v|^{-2q}\|_{L^\infty(B_{|v|/2}(v))} \lesssim  q^3 \vv^{-2q-4}.\]
Therefore,
\[\begin{split}
 \int_{B_{|v|/2}(v)} K_f(v,v') [G' - G] \dd v' &\lesssim \|D_v^2 G\|_{L^\infty(B_{|v|/2}(v))} \int_{B_{|v|/2}(v)} K_f(v,v') \dd v'\\
&\lesssim q^3\vv^{-2q-4} |v|^{2-2s} [f\ast |\cdot|^{\gamma+2s}](v)\\
&\lesssim q^3 \vv^{-2q-2+\gamma},
\end{split}  \]
by Lemma \ref{l:K_upper_bound}. To estimate the second integral in \eqref{e:QsG}, recall that $g(v) \geq \vv^{-2q}$ for all $v$, which implies $G(v') \leq 0$. This leads to
\[
\begin{split}
 \int_{\R^d \setminus B_{|v|/2}(v)} K_f(v,v') [G' - G] \dd v' &\lesssim -\int_{\R^d \setminus B_{|v|/2}(v)} K_f(v,v') G(v) \dd v'\\
 & \lesssim -G(v) |v|^{-2s} (f\ast |\cdot|^{\gamma+2s})(v)\\
 &\lesssim q\vv^{-2q -2 + \gamma},
 \end{split}
\]
using Lemma \ref{l:K_upper_bound} and $-G(v) = |v|^{-2q}-  \vv^{-2q} \lesssim q \vv^{-2q-2}$. Since $\gamma \leq 2$, we finally have 
$Q_s(f,G)(v) \lesssim a_q \vv^{-2q}$. 
  Returning to \eqref{e:good-term1}, we have shown
\begin{equation}\label{e:Kf-integral}
 \iint_{\R^d\times\R^d} K_f(v,v') \vv^{2q} \varphi^2 (\vv^{2q}\langle v'\rangle^{-2q} - 1)\dd v' \dd v \lesssim  \int_{\R^d} \vv^{4q} \varphi^2 (a_q - b_q\vv^{\gamma_+}) \vv^{-2q} \dd v,
\end{equation}
with $b_q \approx q^s$.

Next, for the last term in \eqref{e:coercivity1}, we proceed as in the proof of Lemma \ref{l:commutator} (see \eqref{e:commutator1}) to write
\begin{equation}\label{e:Kf-integral2}
 \frac \eps 2 \iint_{\R^d\times\R^d} K_f(v,v') \vv^{2q} \varphi^2 (\vv^q\langle v'\rangle^{-q} - 1)^2 \dd v' \dd v \leq \frac \eps 2 C \int_{\R^d} \vv^{2q} \varphi^2 \|f\|_{L^1_{\gamma+2s}} \vv^{\gamma_+} \dd v,
\end{equation}
and for $\eps>0$ sufficiently small depending on $q$, \eqref{e:coercivity1} becomes, using \eqref{e:Kf-integral} and \eqref{e:Kf-integral2},
\begin{align*}
\int_{\R^d} Q(f,\varphi) \varphi(v) \langle v \rangle^{2q} \dd v &\leq \frac 1 2 I - \frac \eps {2(1+\eps)} \iint_{\R^d \times \R^d} (\langle v' \rangle^{q} \varphi' - \langle v \rangle^{q} \varphi)^2 K_f(v,v') \dd v' \dd v  \\
& \phantom{= \frac 1 2 I} + \int_{\R^d} (a_q - b_q \langle v \rangle^{\gamma_+}) \langle v \rangle^{2q} \varphi^2 \dd v.
\end{align*}
Here, since $b_q$ goes to $+\infty$ as $q \to \infty$, it absorbs the first term $\frac 1 2 I$ for large enough $q$. We conclude the proof setting $c_q = \eps/(2(1-\eps))$.
\end{proof}

Recall the seminorm $\|\cdot\|_{\dot N^{s,\gamma}}$ defined in \eqref{e:Ns-gamma}. The next lemma makes precise the (expected) fact that $\|\vv^q g\|_{\dot N^{s,\gamma}}$ controls $\|\vv^p g\|_{\dot N^{s,\gamma}}$ when $q>p$: 
\begin{lemma}\label{l:moments}
For any $q>p\geq 0$, one has
\[   \|\vv^q g \|_{\dot N^{s,\gamma}}^2 \geq \frac 1 2 \|\vv^p g\|_{\dot N^{s,\gamma}}^2 - C_{q-p} \|g\|_{L^2_{q+(\gamma+2s-2)/2}}^2.\]
\end{lemma}
\begin{proof} 
We prove the lemma in the case $p=0$. The general case follows by replacing $g$ with $\vv^p g$ and replacing $q$ with $q-p$. 

Writing $(\langle v'\rangle^q g' - \vv^q g)^2 = [ \langle v'\rangle^q (g'-g) + g(\langle v'\rangle^q - \vv^q)]^2$ and using the inequality $(a+b)^2 \geq \frac 1 2 a^2 - b^2$, we have
\[ \begin{split}
\|\vv^q g\|_{\dot N^{s,\gamma}}^2 &= \iint_{\R^d\times\R^d} (\vv \langle v'\rangle)^{(\gamma+2s+1)/2} \frac {(\langle v'\rangle^q g' - \vv^q g)^2}{d(v,v')^{d+2s}} \one_{\{d(v,v')\leq 1\}} \dd v' \dd v\\
&\geq  \frac 1 2 \iint_{\R^d\times\R^d} (\vv \langle v'\rangle)^{(\gamma+2s+1)/2} \frac {\langle v'\rangle^{2q} (g'-g)^2}{d(v,v')^{d+2s}} \one_{\{d(v,v')\leq 1\}} \dd v' \dd v\\
&\quad - \iint_{\R^d\times\R^d} (\vv \langle v'\rangle)^{(\gamma+2s+1)/2} \frac {g(v)^2(\langle v'\rangle^q - \vv^q)^2}{d(v,v')^{d+2s}} \one_{\{d(v,v')\leq 1\}} \dd v' \dd v\\
&\geq \frac 1 2 \|g\|_{\dot N^{s,\gamma}}^2 - \int_{\R^d} g(v)^2 \vv^{(\gamma+2s+1)/2} \int_{\{v': d(v,v') \leq 1\}} \langle v'\rangle^{(\gamma+2s+1)/2} \frac {(\langle v'\rangle^q - \vv^q)^2}{d(v,v')^{d+2s}}\dd v' \dd v,
\end{split}\]
using the crude inequality $\langle v'\rangle^{2q} \geq 1$ in the first term. For the second term, since $|v-v'| \leq d(v,v')$, we have
\[\begin{split}
 \int_{\{v': d(v,v') \leq 1\}} \langle v'\rangle^{(\gamma+2s+1)/2} \frac {(\langle v'\rangle^q - \vv^q)^2}{d(v,v')^{d+2s}}\dd v'
 & \lesssim \vv^{(\gamma+2s+1)/2} \int_{\{v': d(v,v') \leq 1\}} \frac { (q|v|^{q-1})^2 |v'-v|^2}{d(v,v')^{d+2s}} \dd v' \\
 &\lesssim q^2 \vv^{2q-2 + (\gamma+2s+1)/2}  \int_{\{v': d(v,v') \leq 1\}} d(v,v')^{2-2s-d} \dd v'.
 \end{split}\]
 To estimate the last integral, we use the change of variables introduced in \cite{imbert2019global}. For any $v\in \R^d$ with $|v|>2$, define the linear transformation $T_0:\R^d\to \R^d$ by
 \[ T_0(av + w) := \frac a {|v|} v + w \quad \text{ where } w \cdot v = 0, a\in \R.\]
 Letting $E_0 = v+ T_0(B_1)$, we see that $E_0$ is an ellipse centered at $v$ with radius $1/|v|$ in the $v$ direction and $1$ in directions perpendicular to $v$. For $v_1,v_2\in E_0$, the distance $d_a$ is defined by $d_a(v_1,v_2) = |T_0^{-1}(v_1-v_2)|$. From \cite[Lemma A.1]{imbert2019global}, there is a universal constant $c>0$ (by increasing the constant, we can take $c>1$) with
  \begin{equation}\label{e:equiv}
   \frac 1 c d_a(v_1,v_2) \leq d(v_1,v_2) \leq c d_a(v_1,v_2), \quad v_1, v_2 \in E_0.
   \end{equation}
 Since $E_0$ is not exactly a subset of $\{v' : d(v,v') \leq 1\}$, in order to use this equivalence of metrics in the above integral, we must first rescale in the $v'$ variable. We claim that for $c>1$ and $v'$ such that $d(v,v')\leq 1$, there holds 
 \begin{equation}\label{e:d-claim}
d(v, v+ (v'-v)/c) \leq \frac 1 c d(v,v').
\end{equation}
Indeed,  writing $v' - v = \alpha v + w$ with $w\cdot v = 0$ and $|\alpha|\leq 1$ (which follows from $d(v,v')\leq 1$),  a series of calculations shows
\[ \begin{split}
 \left(|v|^2 - |v+(v'-v)/c|^2\right)^2 &= \frac 1 {c^4} \left( (\alpha^2 + 2\alpha c)|v|^2+ |w|^2\right)^2\\
&\leq \frac 1 {c^2} \left( (\alpha^2 + 2\alpha)|v|^2 + |w|^2\right)^2\\
&= \frac 1 {c^2} \left( |v|^2 - |v'|^2\right)^2.
\end{split}\]
The inequality to get to the second line follows from $|\alpha|\leq 1$ and $c>1$. Together with $|v-(v+(v'-v)/c)|^2 = |v'-v|^2/c^2$ and the definition of $d(v,v')$, this implies \eqref{e:d-claim}.

Now,  letting $\tilde v = v + (v'-v)/c$ and using \eqref{e:d-claim}, we have 
\[\begin{split}
 \int_{\{v': d(v,v') \leq 1\}} d(v,v')^{2-2s-d} \dd v' &\leq \int_{\{\tilde v: d(v,\tilde v) \leq 1/c\}} (c d(v,\tilde v))^{2-2s-d}  c^3 \dd \tilde v\\
&\leq  c^{5-2s-d}\int_{\{\tilde v: d_a(v,\tilde v) \leq 1\}} c^{d+2s -2}|T_0^{-1}(v-\tilde v)|^{2-2s-d} \dd \tilde v\\
&= c^{3}|v|^{-1}\int_{B_1(0)}  |\hat v|^{2-2s-d} \dd \hat v\\
&\lesssim |v|^{-1},
\end{split}\]
 where we used \eqref{e:equiv} to get to the second line, and the change of variables $\tilde v = v + T_0\hat v$ (with Jacobian $|v|^{-1}$) to get to the third line.
 
The above inequality was derived under the assumption $|v|>2$.  If $|v|\leq 2$, then 
\[\int_{\{v' : d(v,v') \leq 1\}} d(v,v')^{2-2s-d} \dd v' \leq \int_{\{v' : d(v,v') \leq 1\}} |v-v'|^{2-2s-d} \dd v' \lesssim 1 \lesssim |v|^{-1},\] 
in this case as well.

 Finally, we obtain
 \[ \begin{split}
\|\vv^q g\|_{\dot N^{s,\gamma}}^2 &\geq \frac 1 2 \|g\|_{\dot N^{s,\gamma}}^2 - C q^2 \int_{\R^d} g(v)^2 \vv^{2q + \gamma +2s - 2} \dd v,
\end{split} \]
as claimed.
\end{proof}

Now we are ready to establish the propagation of weighted Sobolev norms:

\begin{lemma}\label{l:sobolev-propagation}
Let $f^h$ be the approximate solution constructed in Section \ref{s:approximate_problem}, with initial data $f_0$. There exist $q_1> q_0>0$ such that if $f_0 \in H^4_{q_0}\cap L^\infty_{q_1}(\T^3\times\R^3)$ and $\|f_0\|_{L^\infty_{q_1}} < 1/2$, then there exists $T_1>0$ depending on $\|f_0\|_{H^4_{q_0}(\T^3\times\R^3)}$ and $\|f_0\|_{L^\infty_{q_1}(\T^3\times\R^3)}$, such that $f^h(t,\cdot,\cdot)\in H^4_{q_0}(\T^3\times \R^3)$ for all $t\in [0,T_1]$.

If, in addition, $f_0$ lies in $H^k_q(\T^3\times\R^3)$ for some $k\geq 4$ and $q>0$, and also $f_0 \in L^\infty_m$ for some $m$ large enough depending on $k$ and $q$, then $f^h(t,\cdot,\cdot) \in H^k_q(\T^3\times\R^3)$ for all $t\in [0,T_1]$, with
\[ \|f^h(t,\cdot,\cdot)\|_{H^k_q(\T^3\times\R^3)}\leq \|f_0\|_{H^k_q(\T^3\times\R^3)} \exp\left(\int_0^t C_{k,q}(s) \dd s\right), \quad t\in [0,T_1],\]
for some integrable functions $C_{k,q}(s)>0$ depending on $k$, $q$, and $\|f_0\|_{L^\infty_m}$. The time $T_1>0$ depends \textbf{only} on $\|f_0\|_{H^4_{q_0}}$ and $\|f_0\|_{L^\infty_{q_1}}$, regardless of $k$ and $q$.
\end{lemma}

We remark that the requirement $k\geq 4$ could possibly be improved, but we do not need to optimize it since we intend to use this lemma to guarantee $f^h$ is $C^\infty$ on some time interval whenever $f_0$ is $C^\infty$.

The point of Lemma \ref{l:sobolev-propagation} is to obtain an interval of time where the functions $f^h$ stay smooth, with bounds independent of $h$. It is important that it is the same time interval $[0,T_1]$ where all the estimates apply, for Sobolev norms of any order and any decay rate. In particular, applying this lemma with $f_0$ in the Schwartz space, we get that the solutions $f^h$ stay in the Schwartz space, uniformly in $h$, while $t \in [0,T_1]$.

\begin{proof}
For $q_1>0$ sufficiently large, we use Lemma \ref{l:q-upper-bound} to find some $T_0>0$ depending on $\|f_0\|_{L^\infty_{q_1}}$ so that the $L^\infty_{q_1}$ norm of $f^h$ is less than $1/2$ in $[0,T_0]$. The time $T_1$ in this proof will be less than or equal to this $T_0$. Therefore, throughout this proof, we will absorb $L^\infty_q$ norms of $f^h$ and $M+f^h$ into constants. We will also use without comment the following equivalence of norms, which follows from standard interpolation inequalities:
\[ \|u\|_{H^{k}_q(\T^3\times\R^3)}^2 \approx \|u\|_{L^2_q(\T^3\times\R^3)}^2 + \sum_{|\alpha| = k} \|\partial^\alpha u\|_{L_q^2(\T^3\times\R^3)}^2,  \quad k\in \mathbb N. \]

We have already established an upper bound for the first term as a consequence of Lemma \ref{l:q-upper-bound} applied to an exponent $q_0$ large enough. We are left to propagate an upper bound for the second term. Note also that for any $m < k$ and $q_0 > 0$, by picking $q_1$ sufficiently large, we interpolate an upper bound for the $\dot H^m_{q_0}$ norm,
\begin{equation} \label{e:sobolev_interpolation}
\|u\|_{\dot H^m_{q_0}} \leq \|u\|_{L_{q_1}^\infty}^{\frac{k-m}k} \|u\|_{\dot H^k}^{\frac mk}.
\end{equation}

The proof will proceed by induction on $k$, but we begin with some estimates that apply in both the base case and the inductive step. The precise inductive hypothesis will be given below.


\medskip

\noindent {\bf General estimates.} For some $k,q>0$, let $\alpha\in \N^{3+3}$ be a multi-index with total order $|\alpha|= k$.  Following the usual strategy for $L^2$-based energy estimates, we differentiate the equation for $f^h$ by $\partial^\alpha$, multiply by $\vv^{2q}\partial^\alpha f^h$, and integrate over $\T^3\times\R^3$.

Differentiating the transport equation by $\partial^\alpha$, we have
\[ \partial_t \partial^\alpha f^h= -2 v\cdot \nabla_x \partial^\alpha f^h - 2\sum_{i=1}^3 \one_{\alpha_{3+i}>0} (\partial^{\alpha-(e_i,-e_i)} f^h),\]
where $e_1 = (1,0,0)$, $e_2=(0,1,0)$, and $e_3 = (0,0,1)$, and $\alpha_{3+i}$ is the index corresponding to differentiation in $v_i$. Multiplying by $\vv^{2q}\partial^\alpha f^h$ and integrating, we have
\[ \begin{split}
 \frac 1 2 \frac d {dt} \iint_{\T^3\times\R^3} \vv^{2q} |\partial^\alpha f^h|^2 \dd v \dd x &\leq 2k \sum_{i=1}^3\iint_{\T^3\times\R^3} \vv^{2q} \one_{\alpha_{3+i}>0} (\partial^{\alpha-(e_i,-e_i)} f^h) \partial^\alpha f^h \dd v \dd x\\
&\lesssim k \|D^k f^h\|_{L^2_q(\T^3\times\R^3)}^2.
\end{split}\]
Adding up, we have shown 
\begin{equation}\label{e:transport-times}
   \frac d {dt} \|D^k f^h\|_{L^2_q(\T^3\times\R^3)}^2 \leq Ck \|D^k f^h\|_{L^2_q(\T^3\times\R^3)}^2, \quad t\in \mathcal T.
\end{equation}
Recall that $f^h$ has a jump discontinuity in time at every $t_i$ with $i$ odd. Let $f^h(t_i-)$ and $f^h(t_i+)$ denote the limit of $f^h(t,\cdot,\cdot)$ as $t\to t_i$ from the left and right respectively. From \eqref{e:mollification} it is clear that $\|f^h(t_i+)\|_{\dot H^k_q(\T^3\times\R^3)} \leq \|f^h(t_i-)\|_{\dot H^k_q(\T^3\times\R^3)}$ for any $k,q\geq 0$. Therefore, \eqref{e:transport-times} implies
\begin{equation}\label{e:i-odd}
 \|f^h(t_{i+1})\|_{\dot H^k_q(\T^3\times\R^3)} - \|f^h(t_i-)\|_{\dot H^k_q(\T^3\times\R^3)}  \leq Ck \int_{t_i}^{t_{i+1}} \|f^h(s)\|_{\dot H^k_q(\T^3\times\R^3)}^2 \dd t, \quad i \text{ odd.}    
 \end{equation}

If $t$ belongs to $\mathcal D$, differentiating \eqref{e:space-homogeneous-part} by $\partial^\alpha$ gives
\begin{equation*}
		\partial^\alpha \partial_t f^h 
		=  2\sum_{\alpha' + \alpha'' = \alpha} Q( \partial^{\alpha'}(M+ f^h), \partial^{\alpha''} f)  + 2\sum_{\alpha'+\alpha'' = \alpha} Q(\partial^{\alpha'} f^h, \partial^{\alpha''} M).
\end{equation*}
Note that $\partial^{\alpha''} M=0$ whenever $\alpha''$ contains any differentation in $x$.

The energy estimate takes the form
\begin{equation}\label{e:energy}
\begin{split}
	\frac{1}{2} \frac d {dt} \iint_{\T^3\times\R^3} \vv^{2q} |\partial^\alpha f^h|^2  \dd v \dd x &= 
2 \sum_{\alpha' + \alpha'' = \alpha} \iint_{\T^3\times\R^3} \vv^{2q} Q( \partial^{\alpha'}(M+f^h), \partial^{\alpha''}  f^h)  \partial^\alpha f^h  \dd v \dd x			\\
 &\quad +  2\sum_{\alpha' + \alpha'' = \alpha} \iint_{\T^3\times\R^3} \vv^{2q} Q( \partial^{\alpha'} f^h,  \partial^{\alpha''} M)   \partial^\alpha f^h \dd v \dd x .
\end{split}
\end{equation}
Starting with the last sum on the right, we apply the commutator estimate Lemma \ref{l:commutator} and the trilinear estimate Proposition \ref{p:trilinear} to each term:
\[ \begin{split}
\int_{\R^3} \vv^{2q} Q&( \partial^{\alpha'} f^h,  \partial^{\alpha''} M)   \partial^\alpha f^h \dd v\\
 &\leq \|\vv^{2q} Q( \partial^{\alpha'} f^h,  \partial^{\alpha''} M) - Q(  \partial^{\alpha'} f^h,  \vv^{2q} \partial^{\alpha''} M)\|_{L^2_{-\gamma_+/2}}\|  \partial^\alpha f^h \|_{L^2_{\gamma_+/2}}\\
&\quad + \int_{\R^3} Q ( \partial^{\alpha'} f^h,  \vv^{2q} \partial^{\alpha''} M)  \partial^\alpha f^h \dd v\\
&\lesssim \|  \partial^{\alpha'} f^h \|_{L^1_{\gamma+2s+2}} \|\vv^{2q} \partial^{\alpha''} M\|_{N^{s,\gamma}} \left( \| \partial^\alpha f^h \|_{L^2_{\gamma_+/2}} + \|  \partial^\alpha f^h \|_{N^{s,\gamma}}\right).
\end{split}\]
For the first factor on the right, we use the fact that $\|\cdot\|_{L^1_{\gamma+2s+2}} \lesssim \|\cdot\|_{L^2_q}$ for $q$ sufficiently large. The middle factor is bounded by some constant depending on $k$ and $q$. For the term $\|\partial^\alpha f^h\|_{N^{s,\gamma}}$, we apply Lemma \ref{l:moments}. We finally have
\[ \begin{split}
\int_{\R^3} \vv^{2q} Q( \partial^{\alpha'} f^h,  \partial^{\alpha''} M)   \partial^\alpha f^h \dd v &\lesssim \|\partial^{\alpha'} f^h\|_{L^2_q} \| \left( \|\vv^q \partial^\alpha f^h\|_{\dot N^{s,\gamma}} + \|\partial^\alpha f^h\|_{L^2_{q+(\gamma+2s-2)/2}}\right)\\
 &\leq C_{k,q} c_q^{-1}\left(\|\partial^{\alpha'} f^h\|_{(L^2_q)_v}^2 +\|\partial^\alpha f^h\|_{(L^2_q)_{v}}^2\right) + \frac{c_q} {4k}\|\vv^q \partial^\alpha f^h\|_{\dot N^{s,\gamma}}^2,
\end{split}\]
by Young's inequality, where we choose $c_q>0$ to match the constant from Lemma \ref{l:magic_negative_term_2}. Integrating in $x$ and summing over $\alpha'$, we find
\begin{equation}\label{e:easy-term}
\sum_{\alpha'+\alpha'' = \alpha} \iint_{\T^3\times\R^3}  \vv^{2q} Q( \partial^{\alpha'} f^h,  \partial^{\alpha''} M)   \partial^\alpha f^h \dd v  \leq C_{k,q} \|f^h\|_{(H^k_q)_{x,v}}^2 + \frac {c_q} 4\|\vv^q \partial^\alpha f^h\|_{L^2_x \dot N^{\gamma,s}_v}^2.
\end{equation}

In the first sum on the right in \eqref{e:energy}, consider first the case where all derivatives fall on $f^h$, i.e. $\alpha'=(0,0,0,0,0,0)$.  This case gives us a coercive negative term. From Lemma \ref{l:magic_negative_term_2}, followed by Corollary \ref{c:GS}, 
\begin{equation}\label{e:good-term-cq}
\begin{split}
 \int_{\R^3} Q(M+f^h,\partial^\alpha f^h) \langle v \rangle^{2q} \partial^\alpha f^h \dd v &\leq
-c_q \left( \iint_{\R^3\times\R^3} | \langle v' \rangle^q (\partial^\alpha f^h)'- \langle v \rangle^q  \partial^\alpha f^h|^2 K_{M+f^h} \dd v' \mathrm{d} v\right)\\
&\quad +  \int_{\R^3} (a_q - b_q \langle v \rangle^{\gamma_+}) \langle v \rangle^{2q} (\partial^\alpha f^h)^2 \dd v\\
&\leq -c_q \|\vv^q \partial^\alpha f^h\|_{\dot N^{s,\gamma}}^2 + a_q \|\partial^\alpha f^h\|_{L^2_q}^2 .
\end{split}
\end{equation}
Note that we could apply Corollary \ref{c:GS} because $M+f^h$ satisfies the inequalities \eqref{e:hydro} by Lemma \ref{l:hydro-near_equilibrium}.

Returning to the first sum on the right in \eqref{e:energy}, when $|\alpha'| \neq 0$, the commutator estimate Lemma \ref{l:commutator} and the trilinear estimate Proposition \ref{p:trilinear} give
\begin{equation*}
\begin{split}
\int_{\R^3}&\vv^{2q} Q( \partial^{\alpha'} (M+f^h), \partial^{\alpha''} f^h)   \partial^\alpha f^h \dd v\\
&\leq \| \vv^{2q} Q(\partial^{\alpha'} (M+f^h), \partial^{\alpha''} f^h) - Q(\partial^{\alpha'} (M+f^h), \vv^{2q} \partial^{\alpha''} f^h)\|_{L^2_{-\gamma_+}} \|\partial^\alpha f^h\|_{L^2_{\gamma_+}} \\
&\quad + \int_{\R^3}  Q(\partial^{\alpha'} (M+f^h),\vv^{2q}\partial^{\alpha''} f^h)   \partial^\alpha f^h \dd v\\
&\lesssim  \|\partial^{\alpha'} (M+f^h)\|_{L^1_{\gamma+2s+2}} \|\vv^{2q} \partial^{\alpha''}f^h \|_{N^{s,\gamma} }\left(\| \partial^\alpha f^h\|_{L^2_{\gamma_+/2}}+ \|  \partial^\alpha f^h\|_{N^{s,\gamma}}\right).
\end{split}
\end{equation*}
Using $\|\cdot\|_{N^{s,\gamma}} \lesssim \|\cdot\|_{\dot N^{s,\gamma}} + \|\cdot\|_{L^2_{(\gamma+2s)/2}}$ and Lemma \ref{l:moments}, this implies
\begin{equation}\label{e:alpha-beta}
\begin{split}
\int_{\R^3}&\vv^{2q} Q( \partial^{\alpha'} (M+f^h), \partial^{\alpha''} f^h)   \partial^\alpha f^h \dd v
\leq C_{k,q}\|\partial^{\alpha'} f^h\|_{L^2_q} \|\vv^{2q}\partial^{\alpha''} f^h\|_{N^{s,\gamma}} \left( \|\partial^\alpha f^h\|_{L^2_q} + \|\vv^q \partial^\alpha f^h\|_{\dot N^{s,\gamma}}\right).
\end{split}
\end{equation}
 Note that the largest $v$ moments have been paired with the lower-order derivative $\partial^{\alpha''} f^h$. 
The expression \eqref{e:alpha-beta} will be estimated differently depending on how the derivatives fall. The analysis differs between the base case and the inductive step.

\medskip

\noindent {\bf Base case:} Let $k=4$. In the middle factor on the right in \eqref{e:alpha-beta}, we apply the inequality 
\begin{equation}\label{e:Nsgamma-Hs}
\|h\|_{N^{s,\gamma}} \lesssim \|h\|_{H^s_{(\gamma+2s)/2}},
\end{equation} 
(see \cite[Equation (13)]{gressman-strain2011sharp}). 
The right side of \eqref{e:alpha-beta} is therefore bounded by a constant times
\[ \begin{split}
I_\alpha &:=  
\| \partial^{\alpha'}f^h\|_{L^2_q} \|\vv^{2q}\partial^{\alpha''} f^h\|_{ H^s_{(\gamma+2s)/2}} \left( \|\partial^\alpha f^h\|_{L^2_q} + \|\vv^q \partial^\alpha f^h\|_{\dot N^{s,\gamma}}\right), 
\end{split}\]

Let $j = |\alpha''|$. The case $j=4$ has been considered above in \eqref{e:good-term-cq}. We begin with the case where $2\leq j \leq 3$, so that $1\leq |\alpha'| \leq 2$.  Integrating $I_\alpha$ in $x$, applying H\"older's inequality, and using Sobolev embedding in $x$ in the first factor (note that the $x$ domain $\mathbb T^3$ is compact), we have
\begin{equation*}
\begin{split}
\int_{\T^3}I_\alpha \dd x &\lesssim \|D_x^2\partial^{\alpha'}f^h\|_{L^2_x (L^2_q)_v} \|\vv^{2q} \partial^{\alpha''} f^h\|_{L^2_x H^s_{(\gamma+2s)/2}} \left(\| \partial^\alpha f^h\|_{(L^2_{q})_{x,v}}+ \| \vv^q \partial^\alpha f^h\|_{L^2_x\dot N^{s,\gamma}_v}\right).
\end{split}
\end{equation*}
For the middle factor in this right-hand side, we apply the interpolation Lemma \ref{l:moment_interpolation}:
\[  \|\vv^{2q} \partial^{\alpha''} f^h\|_{L^2_x H^s_{(\gamma+2s)/2}} \lesssim \| f^h\|_{H^{4}_{x,v}}^{(j+s)/4}\| \vv^m f^h\|_{L^2_{x,v}}^{(4-j-s)/4}, \]
for some $m>0$. The last factor $\| \vv^m f^h\|_{L^2_{x,v}}$ is uniformly bounded, by our $L^\infty_q$ bounds on $f^h$.  We now have
\begin{equation}\label{e:energy2}
\begin{split}
\int_{\T^3} I_\alpha  \dd x &\leq C_q \|D_x^2 \partial^{\alpha'} f^h\|_{(L^2_q)_{x,v}} \| f^h\|_{H^{4}_{x,v}}^{(j+s)/4} \left(\| \partial^\alpha f^h\|_{(L^2_{q})_{x,v}}+ \| \vv^q \partial^\alpha f^h\|_{L^2_x\dot N^{s,\gamma}_v}\right)\\
&\lesssim \| f^h\|_{(H^4_q)_{x,v}}^{(7+s)/4}\left(\| \partial^\alpha f^h\|_{(L^2_{q})_{x,v}}+ \|  \vv^q\partial^\alpha f^h\|_{L^2_x\dot N^{s,\gamma}_v}\right)\\
&\lesssim \|f^h\|_{(H^4_q)_{x,v}}^{(11+s)/4} + \|f^h\|_{(H^4_q)_{x,v}}^{(7+s)/4} \|\vv^q\partial^\alpha f^h\|_{L^2_x \dot N^{s,\gamma}_v}\\
&\leq C_q c_q^{-1}  \|f^h\|_{(H^4_q)_{x,v}}^{2(7+s)/4} + \frac {c_q} {4(4^2)} \|\vv^q\partial^\alpha f^h\|_{L^2_x \dot N^{s,\gamma}_v}^2. 
\end{split} \end{equation}
On the other hand, if $0\leq j \leq 1$, then $|\alpha'|\geq 3$, and we apply H\"older's inequality differently: 
\begin{equation*}
\begin{split}
\int_{\T^3} I_\alpha  \dd x &\lesssim \|\partial^{\alpha'}  f^h\|_{(L^2_q)_{x,v}} \|\vv^{2q}\partial^{\alpha''} f^h\|_{L^\infty_x N^{s,\gamma}_v} \left(\| \partial^\alpha f^h\|_{(L^2_{q})_{x,v}}+ \| \vv^q \partial^\alpha f^h\|_{L^2_x \dot N^{s,\gamma}_v}\right).
\end{split}
\end{equation*}
For the middle factor in this expression, we apply Sobolev embedding plus Lemma \ref{l:moment_interpolation}:
\[ \|\vv^{2q} \partial^{\alpha''} f^h\|_{L^\infty_x H^s_{(\gamma+2s)/2}}  \lesssim  \|\vv^{2q} \partial^{\alpha''} f^h\|_{H^2_x H^s_{(\gamma+2s)/2}}  \lesssim \|f^h\|_{H^4_{x,v}}^{(2+s)/4} \|f^h\|_{(L^2_m)_{x,v}}^{(2-s)/4}.\]
We then have
\begin{equation}\label{e:energy4}
\begin{split}
\int_{\T^3} I_\alpha  \dd x &\leq C_q \|\partial^{\alpha'}  f^h\|_{(L^2_q)_{x,v}} \|f^h\|_{H^4_{x,v}}^{(2+s)/4} \left(\| \partial^\alpha f^h\|_{(L^2_{q})_{x,v}}+ \| \vv^q \partial^\alpha f^h\|_{L^2_x\dot N^{s,\gamma}_v}\right)\\
&\lesssim \| f^h\|_{(H^4_q)_{x,v}}^{(10+s)/4} + \|f^h\|_{(H^4_q)_{x,v}}^{(6+s)/4} \|\vv^q\partial^\alpha f^h\|_{L^2_x\dot N^{s,\gamma}_v}\\
&\leq  C_q c_q^{-1}\|f^h\|_{(H^4_q)_{x,v}}^{2(6+s)/4} + \frac {c_q} {4(4^2)} \|\vv^q\partial^\alpha f^h\|_{L^2_x \dot N^{s,\gamma}}^2.
\end{split}
\end{equation}

We sum over $\alpha'$ in \eqref{e:energy} and obtain, using \eqref{e:energy2}, \eqref{e:energy4}, \eqref{e:easy-term}, and \eqref{e:good-term-cq},
\begin{equation}
\begin{split}
\frac 1 2 \frac d {dt} \|\partial^\alpha f^h\|_{(L^2_q)_{x,v}}^2 &\leq  C_q \|\partial^\alpha f^h\|_{(L^2_q)_{x,v}}^2  - \frac {c_q}2 \|\vv^q \partial^\alpha f^h\|_{L^2_x \dot N^{s,\gamma}_v}^2
 + C \|f^h\|_{(H^4_{q})_{x,v}}^{p} + \frac {c_q} 4 \|\vv^q \partial^\alpha f^h\|_{L^2_x \dot N^{s,\gamma}_v}^2,
\end{split}
\end{equation}
for some $p>0$ depending only on $s$. Taking $q$ sufficiently large and summing over $\alpha$, we have
\begin{equation}\label{e:base-case}
 \frac 1 2 \frac d {dt} \|f^h\|_{(H^4_q)_{x,v}}^2 \leq C_q \|f^h\|_{(H^4_q)_{x,v}}^p - \frac {c_q} 4 \sum_{|\alpha|= 4} \|\vv^q \partial^\alpha f^h\|_{L^2_x \dot N^{s,\gamma}_v}^2, \quad t\in \mathcal D.
 \end{equation}
Let $T_1$ be the maximal time of existence of the ODE $\frac 1 2 y'(t) = C_{q} y^{p/2}(t)$ with $y(0) = \|f_0\|_{H^4_q(\T^3\times\R^3)}$. For $t\leq T_1$, integrate \eqref{e:base-case} from $t_i$ to $t_{i+1}-$ for every $i$ even such that $t_{i+1}\leq t$, and combine with the $k=4$ case of \eqref{e:i-odd} to obtain
\begin{equation}\label{e:H4-bound}
\|f^h(t)\|_{H^4_q(\T^3\times\R^3)}^2 \leq \|f_0\|_{H^4_q(\T^3\times\R^3)}^2 + \int_0^t C_q\|f^h(s)\|_{H^4_q(\T^3\times\R^3)}^p\dd s.   
\end{equation}
 Gr\"onwall's lemma implies that $\|f^h(t)\|_{H^4_q(\T^3\times\R^3)}$ is finite up to $T_1$, for $q>0$ sufficiently large. More precisely, we fix a $q_0>0$ sufficiently large, and find a time $T_1>0$ based on estimate \eqref{e:H4-bound} in the case $q=q_0$. This time depends on the norm of $f_0$ in $H^4_{q_0}$ and in $L^\infty_{q_1}$ for some $q_1$ sufficiently large depending on $q_0$. This establishes the first statement of the lemma.
 

If we estimate $H^4_q$ norms of $f^h$ with $q>q_0$ using the above strategy, the resulting time interval on which $\|f^h(t)\|_{H^4_q}$ is finite may depend on the value of $q$.  To get around this issue, for each $q>0$ we interpolate between $H^4$ and $L^\infty_m$, for $m>q$, to obtain a bound for $f^h$ in $H^3_q$. This third-order estimate depends only on $\|f_0\|_{L^\infty_m}$, and on the above bound on $\|f^h\|_{H^4}$, which is uniform on $[0,T_1]$. In particular, we have upper bounds for $f^h$ in $H^3_q$ that hold true on a time interval independent of $q$:
\begin{equation}\label{e:base1}
 \|f^h(t)\|_{H^3_q(\T^3\times\R^3)} \leq C_{3,q},  \quad t\in [0,T_1], q>0.
 \end{equation}
Furthermore, for $|\alpha|\leq 3$ and $q>0$, the estimate \eqref{e:Nsgamma-Hs} implies 
\[\|\vv^q \partial^\alpha f^h\|_{L^2_{t,x}([0,T_1]\times \T^3, N^{s,\gamma}_v)} \lesssim \|\vv^q \partial^\alpha f^h\|_{L^\infty_t L^2_x([0,T_1]\times\T^3, H^s_{(\gamma+2s)/2})},\]
 so by interpolation (Lemma \ref{l:moment_interpolation}) and our uniform bound for $f^h(t)$ in $H^4_{x,v}$, we have 
\begin{equation}\label{e:base2}
 \sum_{|\alpha|\leq 3} \|\vv^q \partial^\alpha f^h(t)\|_{L^2_{t,x} ([0,T_1]\times \T^3,N^{s,\gamma}_v)} \leq C_{3,q}, \quad q>0.
 \end{equation}
Note that this is the undotted $N^{s,\gamma}_v$ norm. We take \eqref{e:base1} and \eqref{e:base2} as the base case of our induction.

\medskip

\noindent {\bf Inductive step:}  For some $k\geq 4$, assume that for all $q>0$ and with $T_1>0$ the same time given in \eqref{e:base1} and \eqref{e:base2},
\begin{equation}\label{e:induction1}
 \|f^h(t)\|_{H^{k-1}_q(\T^3\times\R^3)} \leq C_{k,q}, \quad t\in [0,T_1],
 \end{equation}
and that
\begin{equation}\label{e:induction2}
\|F_{k-1,q}(t)\|_{L^2([0,T_1])} := \left\| \sum_{|\alpha| \leq k-1} \|\vv^q \partial^\alpha f^h(t)\|_{L^2_{x} N^{s,\gamma}_v}\one_{\mathcal D}(t)\right\|_{L^2([0,T_1])} \leq C_{k,q}.
 \end{equation}
 Note that the upper bounds may depend on $k$ and $q$, and also on $L^\infty_m$ norms of $f_0$, where $m>0$ is sufficiently large depending on $q$. However, the time interval $[0,T_1]$ is always the same for any $k\geq 4$ and $q>0$. 
 

As above, we let $\alpha\in \mathbb N^{3+3}$ have total order $k$, differentiate the equation for $f^h$ by $\partial^\alpha$, and integrate against $\partial^\alpha f^h$. If $t\in \mathcal T$, the estimate \eqref{e:transport-times} applies. If $t\in \mathcal D$, we arrive at the expression \eqref{e:energy} as above, and in light of \eqref{e:easy-term} and \eqref{e:good-term-cq}, it only remains to estimate the terms 
\[ \iint_{\T^3\times\R^3} \vv^{2q} Q(\partial^{\alpha'}(M+f^h),\partial^{\alpha''} f^h) \partial^\alpha f^h \dd v \dd x,\]
with $|\alpha'|\neq 0$. The analysis is similar to the base case, except that we use our estimates in $N^{s,\gamma}_v$ norms from the inductive hypothesis to obtain a smaller exponent in the final upper bound for $\frac d {dt} \|f^h\|_{H^k_q}$. 

We arrive at \eqref{e:alpha-beta} as above. If $2\leq |\alpha''|\leq k-1$, we apply H\"older's inequality and Sobolev embedding as follows:
\begin{equation}\label{e:energy6}
\begin{split}
\iint_{\T^3\times\R^3}&\vv^{2q} Q( \partial^{\alpha'} (M+f^h), \partial^{\alpha''} f^h)   \partial^\alpha f^h \dd v\dd x\\
&\leq C_{k,q}\|\partial^{\alpha'} f^h\|_{L^\infty_x(L^2_{q})_v} \|\vv^{2q} \partial^{\alpha''}f^h \|_{L^2_x N^{s,\gamma}_v }\left(\| \partial^\alpha f^h\|_{(L^2_{q})_{x,v}}+ \| \vv^q \partial^\alpha f^h\|_{L^2_x\dot N^{s,\gamma}_v}\right)\\
&\leq C_{k,q} c_0^{-1} \|f^h\|_{(H^k_q)_{x,v}}^2 F_{k-1,2q}^2(t) + \frac {c_0} {4k^2}\|\vv^q\partial^\alpha f^h\|_{L^2_x \dot N^{s,\gamma}_v}^2,
\end{split}
\end{equation}
where $F_{k-1,2q}(t)$ is defined in \eqref{e:induction2}. 

If $0\leq |\alpha''|\leq 1$, we have
\begin{equation*}
\begin{split}
\iint_{\T^3\times\R^3}&\vv^{2q} Q( \partial^{\alpha'} (M+f^h), \partial^{\alpha''} f^h)   \partial^\alpha f^h \dd v\dd x\\
&\leq C_{k,q}  \|\partial_x^{\alpha'}  f^h\|_{(L^2_q)_{x,v}} \|\vv^{2q}\partial^{\alpha''} f^h\|_{H^2_x N^{s,\gamma}_v} \left(\| \partial^\alpha f^h\|_{(L^2_{q})_{x,v}}+ \|\vv^q  \partial^\alpha f^h\|_{L^2_x\dot N^{s,\gamma}_v}\right)
\end{split}
\end{equation*}
For the middle factor on the right, since $|\alpha''| + 2 \leq 3\leq k-1$, and $\|h\|_{H^2_x N^{s,\gamma}_v} \lesssim \|h + |\nabla_x h| + |D_x^2h| \|_{L^2_x N^{s,\gamma}_v}$ (which follows from the fact that $N^{s,\gamma}$ is a norm in the $v$ variable only) we can also bound this norm with $F_{k-1,2q}(t)$. This leaves us with
\begin{equation}\label{e:energy7}
\begin{split}
\iint_{\T^3\times\R^3}\vv^{2q} &Q( \partial^{\alpha'} (M+f^h), \partial^{\alpha''} f^h)   \partial^\alpha f^h \dd v\dd x \\
&\leq C_{k,q} \|f^h\|_{(H^k_q)_{x,v}} F_{k-1,2q}(t) \left(\| \partial^\alpha f^h\|_{(L^2_{q})_{x,v}}+ \|  \vv^q\partial^\alpha f^h\|_{L^2_x\dot N^{s,\gamma}_v}\right)\\
&\leq C_{k,q} c_q^{-1} \|f^h\|_{(H^k_q)_{x,v}}^2 F_{k-1,2q}^2(t) + \frac {c_q}{4k^2} \|\vv^q \partial^\alpha f^h\|_{L^2_x \dot N^{s,\gamma}}^2.
\end{split}
\end{equation}
Summing over $\alpha'$ in \eqref{e:energy6} and \eqref{e:energy7}, and adding \eqref{e:easy-term} and \eqref{e:good-term-cq} gives
\begin{equation}\label{e:Hk}
 \frac 1 2 \frac d {dt} \|f^h\|_{(H^k_q)_{x,v}}^2 \leq C_{k,q} \|f^h\|_{(H^k_q)_{x,v}}^2 F_{k-1,2q}^2(t) - \frac {c_q} 4 \sum_{|\alpha|= k} \|\vv^q \partial^\alpha f^h\|_{L^2_x \dot N^{s,\gamma}_v}^2, \quad t\in \mathcal D.
 \end{equation}
Combining with \eqref{e:transport-times} and proceeding as in the derivation of \eqref{e:H4-bound}, we conclude
\begin{equation*}
\|f^h(t)\|_{H^k_q(\T^3\times\R^3)}^2 \leq \|f_0\|_{H^k_q(\T^3\times\R^3)}^2 + \int_0^t C_{k,q}F_{k-1,2q}^2(s)\|f^h(s)\|_{H^k_q(\T^3\times\R^3)}^2 \dd s, \quad t\in [0,T_1].     
\end{equation*}
Our inductive hypothesis \eqref{e:induction2} ensures that $F_{k-1,2q}^2$ is integrable on $[0,T_1]$, so Gr\"onwall's lemma implies $\|f^h\|_{H^k(\T^3\times\R^3)}\leq \|f_0\|_{H^k(\T^3\times\R^3)} \exp\left(C_{k,q}\int_0^t F_{k-1,2q}^2(s) \dd s\right)$ for all $t\in [0,T_1]$, as desired. Estimate \eqref{e:Hk} also provides an upper bound on $\sum_{|\alpha|=k} \|\vv^q \partial^\alpha f^h(t)\one_{\mathcal D}(t)\|_{L^2_{t,x} \dot N^{s,\gamma}_v}^2$ that allows us to close the induction.
\end{proof}

\subsection{Convergence of approximate solutions}


Let us start by analyzing the limit $h \to 0$ when the initial data $f_0$ is in the Schwartz space. In that case, from Lemma \ref{l:sobolev-propagation}, we know that there exists a time $T_1>0$, depending on the norms $\|f_0\|_{H^4_q}$ and $\|f_0\|_{L^\infty_q}$ for some $q \in \mathbb N$, so that $f^h(t,\cdot,\cdot)$ is uniformly smooth and rapidly decaying for any $t \in [0,T_1]$. 

We claim that we can extract a subsequence of $f^h$ that converges uniformly over every compact subset of $[0,T_1] \times \T^3 \times \R^3$. The functions $f^h$ are uniformly smooth in $x$ and $v$, but they have jump discontinuities with respect to time at the $t_i$ with $i$ odd (see \eqref{e:mollification}). These jump discontinuities, however, become negligible as $h \to 0$ so that we can apply the Arzela-Ascoli theorem anyway. Indeed, let us consider $0 < \tau_1 < \tau_2 < T_1$. We see that for any $x,v \in \T^3 \times \R^3$
\begin{align*}
f^h(\tau_2,x,v) - f^h(\tau_1,x,v) &= \int_{[\tau_1,\tau_2] \cap \mathcal D} 2Q(M+f^h,M+f^h) \dd t + \int_{[\tau_1,\tau_2] \cap \mathcal T} -2v \cdot \nabla_x f^h \dd t \\
 &+ \sum_{\substack{t_i \in [\tau_1,\tau_2],\\ \text{$i$ odd}}} \left[ \chi_h \ast f^h(t_i-,x,v) - f^h(t_i-,x,v)\right].
\end{align*}

Since $f^h$ is uniformly smooth in $x$ and $v$, we see that $Q(M+f^h,M+f^h)$ and $v \cdot \nabla_x f^h$ are uniformly bounded on compact sets. (In particular, the boundedness of the collision term follows from  Corollary \ref{c:Qsbound} and Lemma \ref{l:convolution-C0}.) Therefore, the two integral terms are bounded by $\lesssim \tau_2-\tau_1$. Moreover, a standard calculation using evenness of $\chi_h$ shows $|\chi_h \ast f^h(t_i-,x,v) - f^h(t_i-,x,v)| \leq h^2 \|D_x^2 f^h\|_{L^\infty}$. Taking into account that there are at most $(\tau_2 - \tau_1)/(2h) + 1$ points $t_i$ with $i$ odd inside $[\tau_1,\tau_2]$, the summation term is bounded by $\lesssim h(\tau_2 - \tau_1) + h^2$, resulting in
\[ |f^h(\tau_2,x,v) - f^h(\tau_1,x,v)| \leq C\left( (1+h)|\tau_1-\tau_2| + h^2 \right). \]
We get a Lipchitz modulus of continuity in time, with a correction term $\lesssim h^2 \to 0$ as $h \to 0$. This is enough to apply Arzela-Ascoli and get that $f^h$ converges uniformly on every compact set to some Lipchitz (in $t$) function $f$ as $h \to 0$ after extracting a subsequence.

For every value of $t \in [0,T_1]$, we have that $f^h(t,\cdot,\cdot)$ is uniformly smooth, uniformly decaying, and converges uniformly over compact sets to $f(t,\cdot, \cdot)$. Therefore, we can upgrade this convergence to convergence in the Schwartz space $f^h(t,\cdot,\cdot) \to f(t,\cdot,\cdot)$ in $\mathcal S$, for every $t \in [0,T_1]$. Moreover, $f$ is smooth with respect to $x$ and $v$ and rapidly decaying as $|v| \to \infty$.

In order to verify that $f : [0,T_1] \times \T^3 \times \R^3 \to \R$ is a solution of \eqref{e:boltzmann-perturbative}, we write the equation in the weak sense and pass to the limit. Indeed, for any $h>0$ and any smooth test function $\varphi = \varphi(t,x,v)$ with compact support in $(0,T_1)\times \T^3\times \R^3$, let $a,b\in (0,T_1)$ be such that $\supp \varphi \subset [a,b]\times\T^3\times\R^3$, and write
\begin{align*} 0 &= \iiint_{[0,T_1] \times \T^d \times \R^d}\left[ -f^h \partial_t \varphi - 2\one_{t \in \mathcal T} f^h v\cdot \nabla_x \varphi  - 2\one_{t \in \mathcal D} Q(M+f^h,M+f^h) \varphi \right]\dd v \dd x \dd t \\
&+ \sum_{\substack{t_i \in [a,b]\\ \text{$i$ odd}}} \iint_{\T^3 \times \R^3} [\chi_h \ast f^h(t_i-,x,v) - f^h(t_i-,x,v)] \varphi(t_i,x,v)  \dd v \dd x\\
& \to \iiint_{[0,T_1] \times \T^3 \times \R^3}\left[ -f \partial_t \varphi - f v\cdot \nabla_x \varphi  - Q(M+f,M+f) \varphi \right]\dd v \dd x \dd t,
\end{align*}
as $h\to 0$, where the collision term converges by the convergence of $f^h\to f$ in $C^m([a,b]\times\T^3\times\R^3)$. (This follows, for example, from Corollary \ref{c:Qsbound} or Lemma \ref{l:q-upper-bound}.)  The second term converges to zero by our earlier estimate $|\chi_h \ast f^h(t_i-,x,v) - f^h(t_i-,x,v)| \leq h^2 \|D_x^2 f^h\|_{L^\infty}$. 


We have constructed a function $f: [0,T_1] \times \T^3 \times \R^3 \to \R$, continuous in all variables, $C^\infty$ smooth in $x$ and $v$, and rapidly decaying as $|v| \to \infty$, and solving \eqref{e:boltzmann-perturbative} in the sense of distributions. By a standard argument, the distributional formulation and continuity in $t$ imply $f$ is differentiable in time, and is therefore a pointwise solution to \eqref{e:boltzmann-perturbative}. 
Differentiating the equation in $t$ arbitrarily many times, we deduce that $f$ is $C^\infty$ smooth in all variables.

We would want to extend this solution for a longer interval of time if possible. We set $f(T_1,\cdot,\cdot)$ as the initial condition and repeat the same construction. We extend the function $f$ to an interval of time $[0,T_2]$ with $T_2-T_1$ depending on the norms $\|f(T_1,\cdot,\cdot)\|_{H^4}$ and $\|f(T_1,\cdot,\cdot)\|_{L^\infty_m}$ for $m \in \mathbb N$. However, according to Theorem \ref{t:reg}, all these norms will be bounded for as long as \eqref{e:hydro} holds. Moreover, from Lemma \ref{l:hydro-near_equilibrium}, we see that it is sufficient to ensure the upper bounds $|f| \leq \langle v \rangle^{-q}/2$. Finally, from Lemma \ref{l:q-upper-bound}, the function $f$ can be extended to an interval of time $[0,T]$ provided only that $\delta e^{C_1 T} \leq 1/2$.

\medskip

\textbf{Rough initial data:} In the general case, $f_0$ is only bounded and measurable, with decay $|f_0(x,v)|\leq \delta \vv^{-q}$.  Let us approximate $f_0$ with a sequence $f^k_0$ of smooth functions with rapid decay as $|v| \to \infty$ (for example by truncation and mollification). For each initial data $f_0^k$, we have a corresponding smooth solution $f^k: [0,T] \times \T^3 \times \R^3 \to \R$. Note that the value of $T$ here is fixed along the sequence, as well as the upper bound $|f^k(t,x,v)| \leq \eps \langle v \rangle^{-q_0}$, since both depend only on the inequality $|f_0(x,v)|\leq \delta \vv^{-q}$.

If $\gamma > 0$,  we also have uniform-in-$k$ regularity and decay estimates for every $f^k$ in $[\tau,T] \times \T^3 \times \R^3$, given by the application of Theorem \ref{t:reg} to $M+f^k$.  Thus, by Arzela-Ascoli, some subsequence of $f^k$ converges in $(C^m\cap L^\infty_q)([\tau,T] \times \T^3 \times \R^3)$, for every $m,q \in \mathbb N$ and every $\tau\in (0,T)$. The limit function $f$ is therefore smooth in $(0,T] \times \T^3 \times \R^3$ and satisfies the inequalities \eqref{e:reg-estimates} stated in Theorem \ref{t:reg}.

If $\gamma \leq 0$, we have the same regularity and decay estimates for $f^k$, from Theorem \ref{t:reg}, but the estimates are only uniform in $k$ up to some level determined by $q$ (the decay exponent of the initial data).  Applying Arzela-Ascoli and passing to the limit as in the previous paragraph,  the limit function $f$ lies in $(C^m\cap L^\infty_q)([\tau,T]\times \T^3\times\R^3)$ for all $\tau \in (0,T)$ and for some $m$ that can be made as large as desired by choosing $q$ large enough.


Now, let $\varphi(t,x,v)$ be a smooth test function with compact support in $[0,T)\times\T^3\times\R^3$. For every $k$,
%
\begin{equation}\label{e:f-k}
\begin{split}
 \iint_{\T^3\times\R^3} f_0^k(x,v) \varphi(0,x,v) \dd x \dd v 
 = \iiint_{[0,T]\times\T^3\times\R^3} \left[ f^k(\partial_t + v\cdot \nabla_x )\varphi + \varphi Q(M+f^k,M+f^k) \right] \dd v \dd x \dd t.
 \end{split}
 \end{equation}
We want to pass to the limit as $k \to \infty$ on both sides of \eqref{e:f-k}. The left hand side is immediate: since $f_0^k \to f_0$ in $L^1_{loc}$, the left-hand side of \eqref{e:f-k} converges to $\iint f_0(x,v) \varphi(0,x,v)\dd x \dd v$.

For any given $t>0$, we have that $f^k(t,\cdot,\cdot) \to f(t,\cdot,\cdot)$ in $(C^m\cap L^\infty_q) (\T^3 \times \R^3)$, for every $m,q \in \mathbb N$. As in the previous subsection, this implies the convergence of the collision term, and we have
\begin{equation}\label{e:fixed-t}
\begin{aligned} \iint_{\T^3\times\R^3} &\left[ f^k(t,x,v) (\partial_t + v\cdot \nabla_x )\varphi + \varphi Q(M+f^k(t,x,v),M+f^k(t,x,v)) \right] \dd v \dd x \\
&\to \iint_{\T^3\times\R^3} \left[ f(t,x,v) (\partial_t + v\cdot \nabla_x )\varphi + \varphi Q(M+f(t,x,v),M+f(t,x,v)) \right] \dd v \dd x.
\end{aligned}
\end{equation}
Since the functions $f^k$ are bounded in $L^\infty_q$ uniformly in $k$, Lemma \ref{l:phiQgh} gives us an upper bound that lets us apply the Dominated Convergence Theorem to integrate \eqref{e:fixed-t} in $t$. Thus, we pass to the limit on both sides of \eqref{e:f-k} and conclude
\begin{equation*}
\begin{split}
 \iint_{\T^3\times\R^3} f_0(x,v) \varphi(0,x,v) \dd x \dd v 
 = \iiint_{[0,T]\times\T^3\times\R^3} \left[ f(\partial_t + v\cdot \nabla_x )\varphi + \varphi Q(M+f,M+f) \right] \dd v \dd x \dd t.
 \end{split}
 \end{equation*}
This is the weak form of the equation \eqref{e:boltzmann-perturbative} with initial data $f_0$, as described in Section \ref{s:weak_solutions}. Here $f$ is $C^\infty$ for positive times if $\gamma >0$,  and $C^m$ for positive times if $\gamma \leq 0$, where $m$ depends on $q$, the decay rate of the initial data.  In either case,  $f$ is a classical solution.

%
%

\section{Existence of global solutions}

\begin{proof}[Proof of Theorem \ref{t:main}]
First, using Proposition \ref{p:short_time}, choose $q>0$ large enough and $\delta_1>0$ small enough so that $\|\langle v \rangle^q f_0\|_{L^\infty(\T^3 \times \R^3)}<\delta_1$ implies the existence of a solution $f$ at least up to time $1$, with initial data $f_0$ and $\|\langle v \rangle^q f(t,\cdot,\cdot)\|_{L^\infty}< \eps_0$ for $t\in [0,1]$. 

If the conclusion of the theorem is false, there is a first time $t_0>1$ where the conclusion fails, i.e. 
\begin{equation}\label{e:contact}
\|\langle v \rangle^q f(t_0,\cdot,\cdot) \|_{L^\infty(\T^3 \times \R^3)} = \eps_0.
\end{equation}

 We claim the time $t_0$ is bounded above, independently of $f_0$. Indeed, since $\|\langle v \rangle^q f(t,\cdot,\cdot)\|_{L^\infty} \leq \eps_0$ on $[0,t_0]$, the function $M+f$ satisfies the hydrodynamic bounds \eqref{e:hydro} by Lemma \ref{l:hydro-near_equilibrium}. Theorem \ref{t:reg} therefore implies $f$ is $C^\infty$ in $(t,x,v)$ and decays faster than any polynomial in $v$, with estimates in $H^k_q(\T^3\times \R^3)$ that are uniform for $t\in [1,t_0]$. By \cite{imbert-mouhot-silvestre-lowerbound2020}, the bounds \eqref{e:hydro} also imply $M+f$ is uniformly bounded below by a Maxwellian $K_0 e^{-A_0|v|^2}$ in $[1,t_0]\times \T^3\times \R^3$, where $K_0$ and $A_0$ depend only on the constants in \eqref{e:hydro}. These facts allow us to apply Theorem \ref{t:global-attractor} to $M+f$. With the choice $p=1$, this gives
\[ \| \vv^q f(t_0,\cdot,\cdot)\|_{L^\infty(\T^3\times \R^3)} \leq C_1 t_0^{-1},\]
where $C_1$ depends only on $\gamma$, $s$, and the constants in \eqref{e:hydro}. This is a contradiction with \eqref{e:contact} if $t_0 > C_1/\eps_0$. 

Next, we apply Proposition \ref{p:short_time} again, with $T = C_1/\eps_0 + 1$, to obtain a $\delta_2>0$ such that if $\|\langle v \rangle^q f_0 \| < \delta_2$, the corresponding solution $f$ exists up to time $T$ with $|f| < \eps_0\vv^{-q}$. Take $\eps_1 = \min\{\delta_1,\delta_2\}$. We have ruled out $t_0>T-1$ and $t_0<T$, so a crossing as in \eqref{e:contact} cannot occur.

We have shown $|f(t,x,v)| < \eps_0\vv^{-q}$ everywhere, which also implies the hydrodynamic bounds \eqref{e:hydro} never degenerate. Combined with the short-time existence result Proposition \ref{p:short_time}, this implies $f$ is a global solution on $[0,\infty)\times\T^3\times\R^3$.
\end{proof}

\begin{proof}[Proof of Theorem \ref{t:main2}] The proof of Theorem \ref{t:main2} follows the same lines as the proof of Theorem \ref{t:main}. The only difference is that the conditional regularity result of Theorem \ref{t:reg} depends on the upper bounds $N_q$ on the initial data, for $q \in \mathbb N$. Thus, we have the same condition on $f_0$ in the statement of Theorem \ref{t:main2}.
\end{proof}

\bibliographystyle{abbrv}
\bibliography{near_maxwellian2}

\end{document}